\documentclass[english]{amsart}

\usepackage{amsmath}
\usepackage{amsthm}
\usepackage{amssymb}
\usepackage{amscd}

\newtheorem{thm}{Theorem}[section]
\newtheorem{lem}[thm]{Lemma}
\newtheorem{cor}[thm]{Corollary}
\newtheorem{prop}[thm]{Proposition}

\newtheorem{conj}{Conjecture}
\theoremstyle{remark} 
\newtheorem{rem}[thm]{Remark}
\newtheorem{ex}[thm]{Example}

\def\Auto{\mathit{Aut}}
\def\bRep{\mathit{Rep}}
\def\gp{\mathfrak p}
\def\gq{\mathfrak q}
\def\gg{\mathfrak g}
\def\gr{\mathfrak r}

\def\gP{\mathfrak P}
\def\gY{\mathfrak Y}
\def\gM{\mathfrak M}
\def\gE{\mathfrak E}

\def\R{\mathcal R}
\def\M{\mathcal M}
\def\8{\mathcal C}
\def\O{\mathcal O}
\def\BT{\mathcal{BT}}
\def\rk{\operatorname{rk}}
\def\id{\operatorname{id}}
\def\Mat{\operatorname{Mat}}
\def\Hom{\operatorname{Hom}}
\def\End{\operatorname{End}}
\def\Gal{\operatorname{Gal}}
\def\Lie{\operatorname{Lie}}
\def\Spec{\operatorname{Spec}}
\def\sy{\operatorname{sym}}
\def\sw{\operatorname{skew}}
\def\sd{\operatorname{std}}
\def\Res{\operatorname{Res}}
\def\SO{\operatorname{SO}}
\def\GO{\operatorname{O}}
\def\OG{\operatorname{GO}}
\def\GL{\operatorname{GL}}
\def\SL{\operatorname{SL}}
\def\PGL{\operatorname{PGL}}
\def\tr{\operatorname{tr}}
\def\ch{\operatorname{char}}
\def\fl{\operatorname{Fil}}
\def\ybar{\overline y}
\def\xbar{\overline x}
\def\Xbar{\overline X}
\def\Fbar{\overline F}
\def\gm{\overline{\mathfrak M}}
\def\Mbar{\overline{\mathcal M}}
\def\fbar{\overline{\mathbb F}}
\def\qbar{\overline{\mathbb Q}}
\def\Phib{\overline{\Phi}}
\def\bark{\overline{\kappa}}
\def\z{\mathbb Z}

\def\a{\mathbb A}
\def\n{\mathbb N}

\def\C{\mathbb C}
\def\f{\mathbb F}
\def\q{\mathbb Q}
\def\r{\mathbb R}
\def\g{\mathbb G}

\def\invlim{\mathop{\vtop{\hbox{\rm lim}\vskip-8pt     
        \hbox{\hskip1pt$\scriptstyle\longleftarrow$}\vskip-1pt}}}

\title{Further counterexamples to Zarhin's conjecture about microweights}
\author{Oliver B\"ultel}

\thanks{14C30, 14G17, 14G35, 14K10, 14L05 (MSC2020), oliver.bultel@bogazici.edu.tr}

\begin{document}
\maketitle

\address{\em Department of Mathematics, Bo\u{g}azi\c{c}i University, 34342 Bebek, \.{I}stanbul, Turkey}

\address{\begin{center}\em Feza G\"ursey Center for Physics and Mathematics,\\Bo\u{g}azi\c{c}i University, 34684 Kandilli, \.{I}stanbul, Turkey\end{center}}

\begin{abstract}
We present two new families of Abelian varieties which contradict Zarhin's conjecture about microweights in positive characteristics. For each 
of these examples we determine the dimension and the Newton-slopes of the ghost Abelian variety in the sense of Cadoret and Tamagawa.
\end{abstract}

\tableofcontents

\section{Introduction}

Abelian varieties over finitely generated fields are amongst the most intensively and frequently studied classical issues in algebraic geometry. Let $\Fbar$ be a separable 
closure of some finitely generated extension $F$ of a prime field of any characteristic, and let $\ell$ be a prime which is invertible in $F$. For an Abelian $g$-fold 
$Y\rightarrow\Spec F$ the $\ell$-adic Tate module $T_\ell Y_{\Fbar}:=\invlim_n Y(\Fbar)[\ell^n]$ (resp. $V_\ell Y_{\Fbar}:=\q\otimes T_\ell Y_{\Fbar}$) is a free $\z_\ell$-module 
(resp. $\q_\ell$-vector space) of rank $2g$. Its significance stems from an extremely interesting $\Gal(\Fbar/F)$-action thereon, and one defines the $\ell$-adic arithmetic 
monodromy group of $Y/F$ to be the smallest $\q_\ell$-algebraic subgroup $G_{Y,\ell}\subset\GL(V_\ell/\q_\ell)$ containing the image of the monodromy representation:
\begin{equation}
\label{monod4}
\rho_{Y,\ell}:\Gal(\Fbar/F)\rightarrow\GL(V_\ell Y_{\Fbar}/\q_\ell)
\end{equation}
The following classical results of Yuri Zarhin are indispensable for further insights into $G_{Y,\ell}$ (cf. \cite{mori}):
\begin{itemize}
\item
$\q_\ell\otimes\End(Y)\cong\End_{G_{Y,\ell}}(V_\ell Y_{\Fbar})$
\item
The algebraic group $G_{Y,\ell}$ is reductive.
\end{itemize}
In the special case $\ch(F)=0$ we may choose an inclusion $\iota$ of $\Fbar$ into $\C$, giving rise to the exponential
$\exp_{Y_\C}:\Lie Y_\C\twoheadrightarrow Y(\C)$ and the period lattice $TY_\C=\ker(\exp_{Y_\C})$ with Hodge decomposition
\begin{equation}
\label{character01}
\C\otimes TY_\C\cong\Lie Y_\C\oplus\check\Lie Y_\C^t,
\end{equation}
where $Y^t$ is the dual Abelian variety. One defines the Mumford-Tate group of $Y_\C$ to be the smallest $\q$-algebraic 
subgroup $H_Y\subset\GL(TY_\C/\z)_\q$ whose scalar extension to $\C$ contains the unique cocharacter 
$$\mu_{Y_\C}:\g_{m,\C}\rightarrow\GL(TY_\C/\z)_\C$$
which renders $\Lie Y_\C$ and $\check\Lie Y_\C^t$ the subspaces of weights $1$ and $0$. It is well-known that the structure of the connected 
group $H_Y$ and its tautological representation $\q\otimes TY_\C$ is severely limited by the existence of a cocharacter with only two weights:
\begin{itemize}
\item[(i)]
The non-abelian simple components of $H_{Y,\qbar}$ are of type $A_n$, $B_n$, $C_n$ or $D_n$.
\item[(ii)]
The irreducible summands of $\qbar\otimes TY_\C$ are (tensor products of) minuscule representations of $H_{Y,\qbar}$.
\end{itemize}
Please see to \cite{deligne3} and \cite{pink2} for more specific results. A natural conjecture, considered and studied by Zarhin, says that (i) and (ii) pertain to hold for the scalar extension to $\qbar_\ell$ 
of the neutral component $G_{Y,\ell}^\circ$ of the subgroup $G_{Y,\ell}\subset\GL(V_\ell Y_{\Fbar}/\q_\ell)$ for any Abelian variety $Y$ over any finitely generated field $F$ of any characteristic different 
from $\ell$ (cf. \cite[Subsection 0.4]{zarhin}). In characteristic $0$ this is motivated by the conjectural equality $G_{Y,\ell}^\circ=H_{Y,\q_\ell}$ (i.e. the Mumford-Tate conjecture). In fact the Zarhin 
conjecture holds in characteristic $0$ by work of Richard Pink (cf. \cite[Corollary 5.11]{pink1}) although the Mumford-Tate conjecture seems to be open, but please see to \cite{deligne2} for the inclusion 
$G_{Y,\ell}^\circ\subset H_{Y,\q_\ell}$. The focus of this note does lie on the case $\ch(F)\notin\{0,2\}$ and its purpose is to construct and to study certain examples for which the derived subgroup 
of the neutral component $G_{Y,\ell}^{der}:=[G_{Y,\ell}^\circ,G_{Y,\ell}^\circ]$ is a non-trivial group of adjoint type. This strongly contradicts Zarhin's conjecture, because the only minuscule representation 
of such groups is the trivial one, given that minuscule representations are already completely determined by their restriction to the center (cf. \cite[Chapter III, Section 13, Exercise 13]{humphreys}).

\begin{thm}
\label{monod7}
For every $p>2$ there exists an Abelian $6$-fold $Y$ over a finitely generated extension of $\mathbb F_p$ such that $G_{Y,\ell}^\circ$ is a $\q_\ell$-form 
of $\g_m\times\SO(3)^2$ for every prime $\ell\neq p$, where the tautological representation of $G_{Y,\ell,\qbar_\ell}^\circ$ can be written as the direct sum 
of two copies of the two projections $\g_m\times\SO(3)^2\rightarrow\g_m\times\SO(3)\rightarrow\GL(3)$ and identifies $\g_m$ with the group of homotheties.
\end{thm}

Our motivation for theorem \ref{monod7} was spurred by work of Oort and van der Put, who applied the Mumford-Faltings-Chai construction for a production of principally polarized Abelian varieties $Y$ of any dimension 
$g\geq5$ over $\f_{p^2}((t))$ such that $\End^0(Y_{\overline{\f_{p^2}((t))}})$ is a quaternion algebra over $\q$, cf. \cite[Theorem(1.1),Example(1.5.1)]{oort1}. Notice that a choice of closed immersion of $Y$ into a suitable 
projective space endows $Y$ with a descent to the subfield of $\f_{p^2}((t))$ which is generated by the coefficients of all equations needed to define it as a subvariety. By doing so, we prove in subsection \ref{fractal7} that 
the cases $g\equiv1\pmod2$ of Oort and van der Put's examples contradict Zarhin's conjecture too (proposition \ref{fractal3}). Section \ref{2nd} reviews a $G_2$-case of a rather general construction in \cite{habil}, 
of Abelian varieties $Y$ with prescribed "$\gr$-adic" monodromy groups and representations, where $\gr\nmid p$ is a prime of an auxiliary CM field acting on $Y$, please see to theorem \ref{G2} for details.\\ 
More recently Anna Cadoret and Akio Tamagawa introduced the ghost of an Abelian variety $Y$ over a finitely generated field $F\supset\f_p$: This is an Abelian variety $\gY$ over an unspecified finite field 
$\f_q\supset\f_p$ with the nice property that the $\Gal(\fbar_p/\f_q)$-representations $V_\ell\gY_{\fbar_p}$ agree with the $\q_\ell$-spaces of $T_\ell$-invariants of the Tate vector spaces $V_\ell Y_{\Fbar}$ for choices 
of $\ell\neq p$ and maximal tori $T_\ell\subset G_{Y,\ell}^{der}$, please see to \cite[Section 6]{cadoret} for details. If $Y$ happens to satisfy Zarhin's conjecture (e.g. in the ordinary case, by \cite[Corollary 4.3.1]{zarhin} 
or \cite[Corollary 6.2]{pink2}), then the Galois representations $V_\ell\gY_{\fbar_p}$ agree with the $\q_\ell$-spaces of $G_{Y,\ell}^{der}$-invariants of the Tate vector spaces $V_\ell Y_{\Fbar}$ for every $\ell$. 
Therefore, the ghost of $Y$ is a good measure for the failure of Zarhin's conjecture. Following a suggestion of Prof. Cadoret, we determine the dimensions and the formal isogeny types of the ghosts of our aforementioned 
$G_2$-examples. At last, please see to \cite[Appendix A]{cadoret} for the discussion of yet another $G_2$-example, which appeared in \cite[Chapter 9]{katz}, (for purposes different from studying Zarhin's conjecture).\\
I thank Prof. Cadoret for much good advice and for pointing out the reference \cite{pink1}. I thank Prof. Ikeda and Prof. Goldring for many conversations on the conjectures \ref{O3} and \ref{ShG2} of section \ref{converse} and 
Claudia Glanemann for encouragement. We hope to rekindle interest and foster awareness about further peculiarities of the case $\ch(F)>0$, as pointed out in \cite{habil} and \cite[Question 2A]{oort2}. 

\section{First Example}

We would like to fix a real quadratic number field $K$ and a quaternion algebra $D$ over $K$ such that $p$ is inert and unramified in $K$ and $D$ splits at all but the two 
archimedean places of $K$. In this section we study polarized Abelian $6$-folds $(Y,\lambda)$ over fields $k$ of characteristic $p>2$, such that there exists a Rosati invariant inclusion:
\begin{equation}
\label{monod8}
\iota:D\hookrightarrow\End^0(Y)
\end{equation}
Recall that the Rosati involution on the endomorphism algebra of a polarized Abelian variety is positive (see \cite[Section 21]{mumford}). It follows that the Rosati involution agrees with the main involution of $D$, 
as a nebeninvolution would stabilize a CM subfield of $D$, but the identity is not a positive involution on such fields. In this section we use Zarhin's theorem to deduce theorem \ref{monod7} from scenarios, 
where equality holds in \eqref{monod8} (i.e. of type III(2) in the terminology of \cite[Paragraph 7.2]{oort3}), and we construct the latter by deforming carefully chosen supersingular cases of \eqref{monod8}, 
followed by a descent to an unspecified finitely generated ground ring contained in $\fbar_p[[t]]$. Throughout most of this section we restrict our attention to triples $(Y,\lambda,\iota)$ for which $\iota^{-1}(\End(Y))$ is 
equal to a fixed maximal order, say $\O_D\subset D$. We write $\O_K$ for the ring of integers of $K$ and $*$ for the main involution on $D$, which preserves $\O_D$ because $*+\id_D$ is the reduced trace of $D$.

\subsection{Some Morita equivalences}
\label{morita1}

Let $\O_D\supset\O_K$ be as above and fix a commutative ring $A$. We let $\O_K-\sy(A)$ be the groupoid consisting of triples $(T,M,\phi)$ where $T$ is an $A$-module, $M$ is an 
$A\otimes\O_K$-module, and $\phi:M\times M\rightarrow T$ is an $A$-bilinear map of $A$-modules satisfying $\phi(x,y)=\phi(y,x)$ and $\phi(\alpha x,y)=\phi(x,\alpha y)$ for all $\alpha\in\O_K$ 
and elements $x$ and $y$ of $M$. We let $\O_D-\sw(A)$ be the groupoid consisting of triples $(T,N,\psi)$ where $T$ is an $A$-module, $N$ is an $A\otimes\O_D$-module, and 
$\psi:N\times N\rightarrow T$ is an $A$-bilinear map of $A$-modules satisfying $-\psi(x,y)=\psi(y,x)$ and $\psi(\alpha x,y)=\psi(x,\alpha^*y)$ for all $\alpha\in\O_D$ and elements $x$ and $y$ of $N$.

\begin{lem}
\label{monod5}
If there exists an isomorphism from the $A\otimes\O_K$-algebra $A\otimes\O_D$ to the $A\otimes\O_K$-algebra $\Mat(2\times2,A\otimes\O_K)$, then there exists an 
equivalence between the categories $\O_D-\sw(A)$ and $\O_K-\sy(A)$ which preserves their forgetful functors to the groupoid of $A$-modules given by $(T,.,..)\mapsto T$.
\end{lem}
\begin{proof}
Since $\O_D-\sw(A)$ depends only on $A\otimes\O_D$, it is enough to construct an equivalence $\O_K-\sy(A)\stackrel{\cong}{\rightarrow}\Mat(2\times2,\O_K)-\sw(A)$, 
which we define by $(M,\phi)\mapsto(N,\psi)$, where $N:=M\oplus M$ and $\psi((x_1,x_2),(y_1,y_2)):=\phi(x_1,y_2)-\phi(x_2,y_1)$.
\end{proof}

\begin{cor}
\label{monod3} 
Let $(T,M,\psi)$ be an object of $\O_D-\sw(Q)$, where $T$ (resp. $M$) is a one-dimensional (resp. finite-dimensional) vector space over a field $Q$ and $\psi$ is non-degenerate. 
Consider the algebraic $Q$-group defined by the $Q$-functor $H(A)$ of automorphisms of the triple $(A\otimes_QT,A\otimes_QM,\psi_A)$ when regarded as an object of $\O_D-\sw(A)$, 
where $\psi_A$ denotes the scalar extension of $\psi$ to a varying $Q$-algebra $A$. If $\ch(Q)\neq2$, then $H$ is a form of $\OG(\frac m2)\times_{\g_m}\OG(\frac m2)$, where 
$2\mid\dim_QM=:m$, $\OG(\frac m2)$ denotes the classical group of orthogonal similarities in $\frac m2$ variables and the map $\OG(\frac m2)\rightarrow\g_m$ is the multiplier character.
\end{cor}

Notice that our assumptions on $\O_D\supset\O_K$ imply the existence of $\z_\ell\otimes\O_K$-algebra isomorphisms $\z_\ell\otimes\O_D\cong\Mat(2\times2,\z_\ell\otimes\O_K)$ for all prime 
numbers $\ell$, including $\ell=p$. It follows that the proof of \ref{monod5} carries over to the following scenario, in which a superscripted "$t$" denotes the Serre-dual of an $\ell$-divisible group.

\begin{lem}
\label{monod6} 
For every scheme $S$ and every prime $\ell$ there exists an equivalence between the following groupoids:
\begin{itemize}
\item[(i)]
$\ell$-divisible groups $H/S$ with a homomorphism $-\psi^t=\psi:H\rightarrow H^t$ and an operation 
$\iota:\z_\ell\otimes\O_D\rightarrow\End(H)$ such that $\psi\circ\iota(\alpha^*)=\iota(\alpha)^t\circ\psi$ for every $\alpha\in\z_\ell\otimes\O_D$
\item[(ii)]
$\ell$-divisible groups $G/S$ with a homomorphism $\phi^t=\phi:G\rightarrow G^t$ and an operation 
$\kappa:\z_\ell\otimes\O_K\rightarrow\End(G)$ such that $\phi\circ\kappa(\alpha)=\kappa(\alpha)^t\circ\phi$ for every $\alpha\in\z_\ell\otimes\O_K$
\end{itemize}
If $(G,\phi,\kappa)$ corresponds to $(H,\psi,\iota)$ under the aforementioned equivalence, then: 
\begin{eqnarray}
\label{O1}
&H\cong G^{\oplus2}&\\
\label{O2}
&\ker(\psi)\cong\ker(\phi)^{\oplus2}&
\end{eqnarray}
\end{lem}
\begin{proof}
The homomorphism $\phi$ (resp. $\psi$) can be regarded as a biadditive morphism from $G\times_SG$ (resp. $H\times_SH$) to $\mu_{\ell^\infty}$, 
so that we can apply lemma \ref{monod5} with $A=\z_\ell$ and $T:=\mu_{\ell^\infty}(R)$ where $R$ runs through the category of $\O_S$-algebras.
\end{proof}

\subsection{On $\z/2\z$-graded symmetric Dieudonn\'e-modules of height $6$}
\label{morita3}

Let us fix a perfect field $k$ of characteristic $p>0$. Recall that a Dieudonn\'e module is a triple $(M,F,V)$ where $M$ is a finitely generated free module over the ring $W(k)$ of $p$-typical 
Witt vectors of infinite length and $F$ and $V$ are commuting additive endomorphisms on $M$ satisfying $^FaF(x)=F(ax)$, $aV(x)=V(^Fax)$ and $F(V(x))=px=F(V(x))$ for any 
$x\in M$ and $a\in W(k)$, where $a\mapsto{^Fa}$ denotes the automorphism of the ring $W(k)$ coming from the absolute Frobenius on $k$. The following facts are well-known:
\begin{itemize}
\item
For every Dieudonn\'e module $(M,F,V)$ the dual $\Hom_{W(k)}(M,W(k))$ inherits a $F$-linear (resp. $F^{-1}$-linear) 
endomorphism from $V$ (resp. from $F$), thus making it a Dieudonn\'e module, which is called the Cartier dual of $M$.
\item
Covariant Dieudonn\'e theory provides an equivalence between the category of $p$-divisible groups over $k$ and the category of Dieudonn\'e modules, 
which interchanges Serre duality and Cartier duality while the height of a $p$-divisible group agrees with the $W(k)$-rank of its Dieudonn\'e module.
\end{itemize}
For every pair of non-negative coprime integers $(a,b)$ there exists a smallest Dieudonn\'e module containing a non-zero element $x$ with $F^a(x)=V^b(x)$. Its 
$p$-divisible group $G_{a,b}$ is of height $a+b$, and the following holds over any algebraically closed ground field $k$: Every $p$-divisible group over $k$ is isogenous 
to a direct sum of isosimple ones and every isosimple $p$-divisible group over $k$ is isogenous to a $p$-divisible group of the form $G_{a,b}$, for a unique pair $(a,b)$.\\
By a $\z/r\z$-gradation on a Dieudonn\'e module $(M,F,V)$ we mean a $\z/r\z$-gradation on $M$ such that $F$ is homogeneous of degree $1$, so that $V$ is homogeneous of degree $-1$. 
If $k$ contains a field of cardinality $p^r$, then a $\z/r\z$-gradation on $M$ is nothing but the eigenspace decomposition of a $W(\f_{p^r})$-operation, where $\f_{p^r}:=\{a\in k\mid a^{p^r}=a\}$. 
In particular $p$-divisible groups $G$ over $k$ with additional structure as indicated in part (ii) of lemma \ref{monod6} correspond to $\z/2\z$-graded Dieudonn\'e modules $M=M_0\oplus M_1$ 
equipped with a non-degenerate symmetric bilinear pairing $M\times M\rightarrow W(k)$ satisfying the usual relation 
\begin{equation}
\label{monod9}
^F(x,V(y))=(F(x),y)
\end{equation}
together with $(x,y)=0$ for any $x\in M_0$ and $y\in M_1$ (provided that $k$ is an algebra over the ring $\O_K$ of that lemma). Let us write $K(k)$ for the field $W(k)[\frac1p]$ and let us say that 
two such $\z/2\z$-graded (symmetric) Dieudonn\'e-modules $M$ and $M'$ are isogenous if there exists a $K(k)$-linear isomorphism $\q\otimes M\stackrel{\cong}{\rightarrow}\q\otimes M'$ 
which preserves (the pairing and) the $\z/2\z$-gradation and commutes with $F$ and $V$. In the sequel we need a slight variant of the so-called skeleton. For a supersingular 
$\z/2\z$-graded Dieudonn\'e module $M=M_0\oplus M_1$ over an algebraically closed extension $k\supset\O_K/p\O_K$ we consider the $K(\f_{p^2})$-vector space:
$$S(M):=\{x\in\q\otimes M_0\mid F(x)=V(x)\}$$
The functor $M\mapsto S(M)$ provides an equivalence between the groupoid of finite dimensional $K(\f_{p^2})$-vector spaces and the groupoid of supersingular 
$\z/2\z$-graded Dieudonn\'e modules up to isogeny. Moreover, a non-degenerate symmetric pairing on $M$ in the previous sense restricts to a non-degenerate pairing 
$$S(M)\times S(M)\rightarrow K(\f_{p^2}),$$ 
as \eqref{monod9} implies $^{F^2}(x,y)=(\frac{F^2(x)}p,\frac{F^2(y)}p)=(x,y)$ for any $x$, $y\in S(M)$. Whence it follows that the groupoid of isometry classes of finite dimensional 
$K(\f_{p^2})$-vector spaces with non-degenerate symmetric pairing is equivalent to the groupoid of supersingular symmetric $\z/2\z$-graded Dieudonn\'e modules up to isogeny. This will 
prove useful for the construction of supersingular cases of \eqref{monod8} with a prescribed $\z/2\z$-graded symmetric Dieudonn\'e module (lemma \ref{morita2} and corollary \ref{morita5}). 

\begin{lem}
\label{morita6}
The formal isogeny type of a $p$-divisible group $G$ of height $6$ (resp. $4$) with additional structure as described in part (ii) of lemma 
\ref{monod6} is $G_{0,1}^2\oplus G_{1,1}\oplus G_{1,0}^2$ or $G_{1,1}^3$ (resp. $G_{0,1}^2\oplus G_{1,0}^2$ or $G_{1,1}^2$).
\end{lem}
\begin{proof} 
If a specimen $G$ of height $6$ did not contain any copies of $G_{1,1}$ it would be isogenous to a group of the form $\tilde G\oplus\tilde G^t$ where $\tilde G$ can be taken 
to be isotypic of height $3$ (in fact isogenous to $G_{0,1}^3$ or $G_{1,2}$). This is impossible because the $\z_p\otimes\O_K$ operation preserves the isotypic group $\tilde G$, 
but $3$ is odd, so that its Dieudonn\'e module cannot have a $\z/2\z$-grading. This reduces the lemma to the assertion for specimens of height $4$, which is easy.
\end{proof}

\begin{ex}
\label{morita4}
The following example of a supersingular $\z/2\z$-graded symmetric Dieudonn\'e module of rank $6$ will play a crucial role for our intended deformations of \eqref{monod8}. Consider a perfect field $k$ 
containing a field of cardinality $p^2$ and put $M_\sigma:=W(k)x_\sigma\oplus W(k)y_\sigma\oplus W(k)z_\sigma$ for $\sigma\in\{0,1\}$ and define $F(x_0)=x_1$, $F(y_0)=y_1$, $z_0=V(z_1)$, $F(x_1)=z_0$, 
$y_1=V(y_0)$, $z_1=V(x_0)$ and $(y_0,y_0)=(x_0,z_0)=(x_1,z_1)=1$ and $(y_1,y_1)=p$ and $(z_\sigma,z_\sigma)=(z_\sigma,y_\sigma)=(x_\sigma,y_\sigma)=(x_\sigma,x_\sigma)=0$ for $\sigma\in\{0,1\}$. 
\end{ex}

\subsection{Deformations}

We proceed to the construction of a deformation, which is based on the Serre-Tate theorem \cite{katz2}.

\begin{lem}
\label{morita2}
For $\O_D\supset\O_K$ as above and any algebraically closed extension $k\supset\O_K/p\O_K$, every finite dimensional $K(\f_{p^2})$-vector space with non-degenerate symmetric pairing arises as a skeleton from 
the supersingular $\z/2\z$-graded symmetric Dieudonn\'e-module associated to a supersingular polarized Abelian variety $(Y_0,\lambda_0)$ with Rosati invariant operation $\iota_0:\O_D\rightarrow\End(Y_0)$ over $k$.
\end{lem}
\begin{proof}
By taking products it is enough to look after triples $(Y_0,\lambda_0,\iota_0)$ where $Y_0$ is a surface. Let $\gE$ be a supersingular elliptic curve over $k$ and consider the functorial tensor 
product $\O_K\otimes\gE$, in the sense of \cite[Chapitre IX, Subsection 1.2]{mori}. According to \cite[Chapitre IX, Subsection 1.3]{mori}, we may choose a polarization on $\O_K\otimes\gE$ to 
arise from the product of the positive definite trace form on $\O_K$ with the unique principal polarization $\lambda_\gE$ of $\gE$. Once an isomorphism between $K\otimes\End(\gE)$ and $D$ 
is fixed, we obtain a Rosati invariant operation $D\rightarrow\End^0(\O_K\otimes\gE)$. Let $(Y_0,\lambda_0,\iota_0)$ be an $\O_D$-invariant member of the isogeny class of $\O_K\otimes\gE$ 
(N.B.: $\O_K\otimes\End(\gE)\ncong\O_D$). Notice that every totally positive $\alpha\in\O_K$ gives rise to another legitimate triple namely $(Y_0,\lambda_0\circ\alpha,\iota_0)$ whose skeleton 
is the one of $(Y_0,\lambda_0,\iota_0)$ multiplied with $\alpha$. We are done, since every element in $(\q_p\otimes K)^*/((\q_p\otimes K)^*)^2$ has a totally positive representative in $\O_K$.
\end{proof}

\begin{cor}
\label{morita5}
For $\O_D\supset\O_K$ as above and any algebraically closed extension $k\supset\O_K/p\O_K$, there exists a polarized Abelian $6$-fold $(Y_0,\lambda_0)$ with Rosati 
invariant operation $\iota_0:\O_D\rightarrow\End(Y_0)$ over $k$ whose associated $\z/2\z$-graded symmetric Dieudonn\'e module is the example described in \ref{morita4}.
\end{cor}
\begin{proof}
The example \ref{morita4} of a $\z/2\z$-graded symmetric Dieudonn\'e module  is clearly associated to some triple $(G,\phi,\kappa)$ consisting of a $p$-divisible group $G$ with additional
structure as indicated in part (ii) of lemma \ref{monod6}. Let $(H,\psi,\iota)$ be its pendant in the sense of the equivalence which is described there. By lemma \ref{morita2} there exists a triple 
$(Y_0,\lambda_0,\iota_0)$ together with a quasi-isogeny $Y_0[p^\infty]\stackrel{\eta}{\dashrightarrow}H$ which preserves the $\O_D$-action and the Weil-pairing. For sufficiently large $n$ one may 
consider the Abelian variety $Z_0:=Y_0/\ker(p^n\eta)$, which is naturally equipped with an $\O_D$-action. Looking at $\psi$ reveals that $p^{2n}\lambda_0$ descends to $Z_0$ and we are done.
\end{proof}

For every perfect field $k$ of characteristic $p>2$ there exists a convenient description of the category of $p$-divisible groups of some finite height $h$ and of some dimension $d\in\{0,\dots,h\}$ over $k[[t]]$ 
in terms of so-called Dieudonn\'e displays in the sense of \cite{zink3}. To this end one must introduce the subring $\hat W(k[[t]]):=W(k)\oplus\hat W(tk[[t]])$ of $W(k[[t]])$, where $\hat W(tk[[t]])$ consists 
of Witt vectors whose components $(x_0,x_1,\dots)$ satisfy $tk[[t]]\ni x_i\rightarrow0$ in the $t$-adic topology (cf. paragraph 2 of loc.cit.). A Dieudonn\'e display is a quadruple $(P,Q,F,V^{-1})$ where $P$ is a 
free $\hat W(k[[t]])$-module of rank $h$, $Q\subset P$ is a submodule such that $P/Q$ is a free $k[[t]]$-module of rank $d$, $V^{-1}:Q\rightarrow P$ is an $F$-linear homomorphism whose image generates 
$P$ as a $\hat W(k[[t]])$-module, and $F:P\rightarrow P$ is an $F$-linear homomorphism satisfying $V^{-1}({^Va}\cdot x)=aF(x)$ for all $a\in\hat W(k[[t]])$ and $x\in P$. Zink's results imply that the category 
of $k[[t]]$-Dieudonn\'e displays is equivalent to the category of $p$-divisible groups over $k[[t]]$. The Dieudonn\'e display of the Serre dual is given by $(P^*,Q^\perp,F,V^{-1})$, whose underlying modules are
$$\Hom_{\hat W(k[[t]])}(P,\hat W(k[[t]]))=:P^*\supset Q^\perp:=\{x\in P^*\mid\,\forall y\in Q:\,(x,y)\in\ker(w_0)\},$$
where $(x,y)$ stands for the perfect pairing between $P$ and its dual and $w_0$ is the projection onto $k[[t]]$ defined by $(x_0,x_1,\dots)\mapsto x_0$. The Dieudonn\'e display structure is set up, by requiring, that 
$F:P^*\rightarrow P^*$ (resp. $V^{-1}:Q^\perp\rightarrow P^*$) must satisfy $^F(x,y)=(F(x),V^{-1}(y))$ (resp. $^V(V^{-1}(x),V^{-1}(y))=(x,y)$) for any element $y\in Q$ and any $x\in P^*$ (resp. $x\in Q^\perp$).\\ 
Over more general $p$-adically complete ground rings, there exists a parallel theory of "nilpotent displays", which has the disadvantage 
that it is only well-behaved for the subcategory of formal $p$-divisible groups, please see to \cite{zink3} for details.

\begin{prop}
\label{NOW}
Let $\O_D\supset\O_K$ be as above and consider a polarized Abelian $6$-fold $(Y_0,\lambda_0)$ with Rosati invariant operation $\iota_0:\O_D\rightarrow\End(Y_0)$ over a perfect field 
extension $k\supset\O_K/p\O_K$. If its associated $\z/2\z$-graded symmetric Dieudonn\'e module agrees with the (supersingular) example \ref{morita4}, then there exists a deformation 
$(Y,\lambda,\iota)$ over $k[[t]]$ of $(Y_0,\lambda_0,\iota_0)$, such that the formal isogeny type of the generic fiber of $Y[p^\infty]$ is $G_{0,1}^4\oplus G_{1,1}^2\oplus G_{1,0}^4$.
\end{prop}
\begin{proof}
According to the Serre-Tate theorem and lemma it \ref{morita6}, it is enough to construct a non-supersingular deformation of the example \ref{morita4}. We start out 
from a change of base $W(k)\rightarrow\hat W(k[[t]])$, which yields a $\z/2\z$-graded symmetric Dieudonn\'e display $(P_0\oplus P_1,Q_0\oplus Q_1,F,V^{-1})$, where 
$P_\sigma=\hat W(k[[t]])\otimes_{W(k)}M_\sigma$ while $Q_\sigma$ is the kernel of $P_\sigma\rightarrow k[[t]]\otimes_kM_\sigma/VM_{\sigma+1}$ and similarly for $F$ and $V^{-1}$. Now let us write 
$(t,0,\dots)=:[t]\in\hat W(tk[[t]])$ for the Teichm\"uller lift of the element $t\in k[[t]]$ and consider the $\z/2\z$-graded $\hat W(k[[t]])$-linear automorphism $U$ on $P:=P_0\oplus P_1$ which is given by:
\begin{eqnarray*}
&&x_0\mapsto x_0\\
&&y_0\mapsto y_0+[t]x_0\\
&&z_0\mapsto z_0-[t]y_0-\frac{[t]^2}2x_0\\
&&x_1\mapsto x_1\\
&&y_1\mapsto y_1\\
&&z_1\mapsto z_1
\end{eqnarray*}
The $\z/2\z$-gradation and (the prolongation to $P_0\oplus P_1$ of) the symmetric pairing are preserved by $U$. To obtain a non-trivial deformation $(P_0\oplus P_1,Q_0\oplus Q_1,\tilde F,\tilde V^{-1})$ 
we precompose the maps $F$ and $V^{-1}$ with $U$ (N.B.: The composition of the linear map $U$ with the $F$-linear ones $F$ and $V^{-1}$ yield $F$-linear maps $\tilde F$ and $\tilde V^{-1}$). Due to
$\tilde F^2(x_1)=\tilde F(z_0-[t]y_0-\frac{[t]^2}2x_0)=pz_1-[t]^py_1-\frac{[t]^{2p}}2x_1$ and $pz_1-[t]^py_1\in Q_1$ we know that the generic fiber of our deformation is not supersingular. 
\end{proof}

\begin{lem}
\label{THEN}
For every $\O_D\supset\O_K$ as above and every polarized Abelian $6$-fold $(Y,\lambda)$ with Rosati invariant operation 
$\iota:\O_D\rightarrow\End(Y)$ over an algebraically closed field $k$ of characteristic $p>0$, one of the following assertions holds:
\begin{enumerate}
\item
$Y$ has complex multiplication.
\item
$\End(Y)$ is isomorphic to $\O_D$.
\end{enumerate}
Moreover, for $p\neq2$ and sufficiently large $k$ there exist triples $(Y,\lambda,\iota)$ for which the latter case holds.
\end{lem}
\begin{proof}
Recall that over an algebraically closed field of positive characteristic the property of having complex multiplication is equivalent to being isogenous to an Abelian variety definable over 
a finite field. So let us consider a triple $(Y,\lambda,\iota)$ which satisfies none of the two assertions above. Since supersingular Abelian varieties do have complex multiplication, we 
know that the formal isogeny type of $Y[p^\infty]$ must be $G_{0,1}^4\oplus G_{1,1}^2\oplus G_{1,0}^4$. We claim that $Y$ is (absolutely) simple. The occurrence of two different 
isogeny factors would lead straightforwardly to a decomposition $X\times_kZ$ each of whose factors is acted on by certain orders of $D$, where $\dim_kX=2$ and $\dim_kZ=4$. The 
formal isogeny type of $X$ (resp. $Z$) must be $G_{1,1}^2$ (resp. $G_{0,1}^4\oplus G_{1,0}^4$). We deduce that $X$ has complex multiplication. However the ordinary Abelian $4$-fold 
$Z$ with $D$-action has complex multiplication too, as one can see from the theory of canonical lifts and the analogous fact in characteristic $0$. Having ruled out the $X\times_kZ$-case 
we proceed to whether or not $Y$ could be a power of a single simple isogeny factor $Z$, and looking at the formal isogeny type leaves no possibility but $Y$ being isogenous to $Z\times_kZ$. 
As observed in \cite[Paragraph 7.2]{oort3}, the endomorphism algebra of an (absolutely) simple Abelian solid cannot be a definite quaternion algebra (i.e. of type III(1), when using the 
notation of loc.cit.). In fact all possible endomorphism types can be read off from the classification which is given there: So $\End^0(Z)$ is either equal to $\q$ or isomorphic to a totally 
real cubic field or a (skew) field extension of degree $2$, $6$ or $18$ over $\q$, provided that it possesses a positive involution of the second kind (i.e. of type I(1), I(3), IV(1,1), IV(3,1) or 
IV(1,3)). Again we would obtain a contradiction, since $D$ cannot be accommodated in $\Mat(2\times2,\End^0(Z))$ in the first three cases, while $Z$ is of CM-type in the last two cases.\\
Since we checked the simplicity of $Y$, we know that $B:=\End^0(Y)$ is a skew-field. Let $H$ be a maximal commutative sub-algebra of $D$, which is a quartic field containing $K$. Extending 
to a maximal commutative sub-algebra of $B$ yields a field of degree $4$ or $12$, of which the latter is ruled out by our assumption that $Y$ was not of CM type. 
We deduce that $H$ remains maximal commutative in $B$, so that the center of $B$ is contained in $K$. If $B$ was strictly bigger than $D$ its center would 
be nothing but $\q$, so that the former is just a form of $\Mat(4\times4,\q)$. Its invariant is contained in the $2$-torsion of the Brauer group of $\q$, given that the 
Rosati involution is an isomorphism between $B$ and $B^{opp}$. It follows that $B$ has the shape $\Mat(2\times2,B')$, which contradicts with $B$ being a skew-field.\\
The occurrence of the latter case (2) is granted at least over $\overline{\f_{p^2}((t))}$ by our proposition \ref{NOW}.
\end{proof}

The previous lemma implies the theorem \ref{monod7} of the introduction. Some of the results in this subsection were announced in 
my talk \cite{height3/2}, which was inspired by a problem of Oort on whether every positive rational number can be written in the form 
\begin{equation}
\label{fractal2}
\frac{2\dim Y}{[\End^0(Y):\q]}
\end{equation}
where $Y$ runs through all simple Abelian varieties over algebraically closed extensions of $\f_p$ (N.B.: If the characteristic was zero, then \eqref{fractal2} would be 
a natural number, namely the dimension of the rational period lattice $\q\otimes TY_\C$ as a vector space over the skew-field $\End^0(Y)$, cf. \cite[Question 2A]{oort2}). 
The lemma \ref{THEN} of this section gives a solution for the number $\frac32$ whereas \cite[Example(1.5.1)]{oort1} gives solutions for any $\frac g2$ with $5\leq g\in\n$.

\subsection{Conclusions}
\label{envelope}

Fix $\Fbar\supset F\supset\f_p$ as in the introduction and let $\O_D\supset\O_K$ be as above and let $(Y,\lambda)$ be a 
polarized Abelian $6$-fold over $F$ which is equipped with a Rosati invariant action $\iota:\O_D\rightarrow\End(Y)$. Let us write 
$$\psi_\ell:V_\ell Y_{\Fbar}\times V_\ell Y_{\Fbar}\rightarrow\q_\ell(1):=V_\ell\g_{m,\Fbar}$$ 
for the Weil-pairing and $H_\ell\subset\GL(V_\ell Y_{\Fbar}/\q_\ell)$ for the $\q_\ell$-subgroup defined by the $\q_\ell$-functor $H_\ell(A)$ of automorphisms 
of the triple $(A\otimes_{\q_\ell}\q_\ell(1),A\otimes_{\q_\ell}V_\ell Y_{\Fbar},\psi_{\ell,A})$ when regarded as an object of $\O_D-\sw(A)$, where $\psi_{\ell,A}$ 
denotes the scalar extension of $\psi_\ell$ to a $\q_\ell$-algebra $A$. In the case at hand \eqref{monod4} factors through a homomorphism
\begin{equation}
\label{monod1}
\Gal(\Fbar/F)\rightarrow H_\ell(\q_\ell)
\end{equation}
so that we may regard the $\ell$-adic arithmetic monodromy group of $Y/F$ as a subgroup $G_\ell\subset H_\ell$. Let $H_\ell^\circ$ be 
the neutral component of $H_\ell$. Corollary \ref{monod3} and the non-degeneracy of $\psi_\ell$ imply that $H_{\ell,\qbar_\ell}^\circ$ 
is isomorphic to $\g_m\times\SO(3)^2$, in particular the derived subgroup of $H_\ell^\circ$ is of adjoint type, since it is a form of $\SO(3)^2$.

\begin{lem}
\label{lorentz}
Let $\O_D\supset\O_K$ be as above and consider a polarized Abelian $6$-fold $(Y,\lambda)$ with Rosati invariant 
$\O_D$-operation $\iota$ over some finitely generated field $F$ of characteristic $p>0$. Then, one of the following assertions holds:
\begin{enumerate}
\item
The neutral component of the $\ell$-adic arithmetic monodromy group of $Y/F$ is a torus.
\item
The neutral component of the $\ell$-adic arithmetic monodromy group of $Y/F$ is equal to $H_\ell^\circ$ 
and the ghost (in the sense of \cite[Section 6]{cadoret}) of $Y$ is a supersingular Abelian surface.
\end{enumerate}
Moreover, for $p\neq2$ and sufficiently large $F$ there exist triples $(Y,\lambda,\iota)$ for which the latter case holds.
\end{lem}
\begin{proof}
In view of \cite{mori} we have:
\begin{equation}
\label{monod2}
\q_\ell\otimes\End(Y_{\Fbar})\cong\End_{G_\ell^\circ}(V_\ell Y_{\Fbar})
\end{equation}
In particular $Y_{\Fbar}$ is of CM type if and only if $G_\ell^\circ$ is a torus. In order to determine $G_\ell^\circ$ in the non-CM case, we have to establish the surjectivity of the two projections 
$G_{\ell,\qbar_\ell}^\circ\rightarrow\g_m$ and $G_{\ell,\qbar_\ell}^\circ\rightarrow\SO(3)^2$. The former is clear from $\q_\ell(1)\not\cong\q_\ell$ and to do the latter we may assume that $G_{\ell,\qbar_\ell}^\circ$ 
was conjugated to a subgroup of $\g_m\times\SO(3)$ or $\SO(3)\times\g_m$ or the diagonal $\SO(3)\subset\SO(3)^2$. However, $\qbar_\ell\otimes_{\q_\ell}\End_{G_\ell^\circ}(V_\ell Y_{\Fbar})$ would be 
isomorphic to $\Mat(2\times2,\qbar_\ell)^4$ or $\Mat(4\times4,\qbar_\ell)$ in these three cases, thus contradicting \eqref{monod2} as $\qbar_\ell\otimes\End(Y_{\Fbar})\cong\Mat(2\times2,\qbar_\ell)^2$. 
Now let $\gY\rightarrow\Spec\f_q$ be (a model of) the ghost of $Y$ in the sense of \cite[Section 6]{cadoret}, where $q$ is a sufficiently big power of $p$, so that $\End^0(\gY_{\fbar_p})=\End^0(\gY)=:B$. 
Proceeding to the structure of $\gY$ we choose maximal tori $T_\ell\subset G_\ell^{der}$, for each prime $\ell\neq p$. Subsection \ref{morita1} shows that the $G_\ell$-representation $V_\ell Y_{\Fbar}$ 
gives rise to a four-dimensional space of $T_\ell$-invariants, because any maximal torus of $\SO(3)$ fixes a one-dimensional subspace in its standard representation. The theorem of Tate and
$V_\ell\gY_{\fbar_p}\cong V_\ell Y_{\Fbar}^{T_\ell}$ proves that the dimension of $\gY$ is two. Moreover, $\gY$ must be a supersingular Abelian surface, according to \cite[Proposition 20]{cadoret}. 
\end{proof}

\subsection{On work of Oort and van der Put}
\label{fractal7}

In this subsection we explain how Zarhin's theorem together with the examples of \cite{oort1} can be used for a short disproof of Zarhin's conjecture for all $p\notin\{0,\ell\}$, albeit 
without giving an explicit description of the $\ell$-adic arithmetic monodromy groups of these counterexamples. We have to begin with a lemma on self-dual minuscule representations:

\begin{lem}
\label{fractal1}
Let $G$ be a semisimple connected algebraic group $G$ over an algebraically closed field $C$ of characteristic $0$. Let $\rho:G\rightarrow\GL(V/C)$ be a minuscule representation of $G$ on a vector space $V$ of 
finite dimension over $C$ and let $V^G\subset V$ denote the subspace which is fixed by $G$. If $(\rho,V)$ is isomorphic to its dual representation on $V^*:=\Hom_C(V,C)$, then $\dim_CV^G\equiv\dim_CV\pmod2$.
\end{lem}
\begin{proof}
Any self-dual representation $V$ can be written as $V_1\oplus\dots\oplus V_r\oplus W\oplus W^*$ for irreducible self-dual representations $(\rho_1,V_1)$, ..., $(\rho_r,V_r)$ and another auxiliary 
representation of $G$ on $W$. Moreover, $V$ is minuscule if and only if all of $V_1$, ..., $V_r$ and $W$ are minuscule. It remains to show that $\frac12\dim_CV_i\in\{\frac12\}\cup\n$ holds for each $i$: Notice 
that the restriction of $\rho_i$ to the center $Z\subset G$ induces a character $\chi_i:Z\rightarrow\g_{m,C}$, due to the irreducibility if $V_i$. Since $(\rho_i,V_i)$ is minuscule we have $\chi_i\neq1$ unless 
$\dim_CV_i=1$, but the presence of a non-degenerate $G$-invariant pairing $\phi_i:V_i\stackrel{\cong}{\rightarrow}V_i^*$ forces the image of $\chi_i$ to be contained in the subgroup $\{1,-1\}\subset\g_{m,C}$. 
At last, notice that the semisimplicity of $G$ implies $\rho(G)\subset\SL(V/C)$, so that $\chi_i^{\dim_CV_i}$ is trivial. We infer that $\dim_CV_i=1$ holds if and only if $\dim_CV_i\equiv1\pmod2$ and we are done.
\end{proof}

\begin{prop}
\label{fractal3}
Let $Y$ be an abelian variety of dimension $1\neq g\equiv1\pmod2$ over a finitely generated extension $F\supset\f_p$ and let $\Fbar$ be the separable closure 
of $F$. If $\End^0(Y_{\Fbar})$ is a quaternion algebra over $\q$, then $Y$ is a counterexample to the Zarhin conjecture (with respect to any $\ell\neq p$).
\end{prop}
\begin{proof}
Once an isomorphism $\qbar_\ell\otimes\End^0(Y_{\Fbar})\cong\Mat(2\times2,\qbar_\ell)$ has been chosen one can write the $G_{Y,\ell,\qbar_\ell}^\circ$-representation $\qbar_\ell\otimes_{\z_\ell}T_\ell Y_{\Fbar}=V$ 
as $W\oplus W$, for some $G_{Y,\ell,\qbar_\ell}^\circ$-representation $W$ of odd dimension $g$. Zarhin's theorem implies the irreducibility of $W$. The existence of a polarization on $Y$ implies $V\cong V^*(1)$, 
so that the Jordan-H\"older theorem allows to deduce $W\cong W^*(1)$ from $W\oplus W\cong W^*(1)\oplus W^*(1)$. Furthermore, $g\neq1$ and $G_{Y,\ell,\qbar_\ell}^{der}\lhd G_{Y,\ell,\qbar_\ell}^\circ$ allow 
to infer that $W$ cannot have any non-zero $G_{Y,\ell,\qbar_\ell}^{der}$-invariants, so that lemma \ref{fractal1} implies that the self-dual $G_{Y,\ell,\qbar_\ell}^{der}$-representation $W$ cannot be minuscule.
\end{proof}

\section{Second example}
\label{2nd}
In this section we obtain the existence of $p$-principally polarized non-CM Abelian $7\cdot8=56$-folds over fields of characteristic $p\notin\{0,2\}$, such that their $\ell$-adic 
geometric monodromy groups are certain $\q_\ell$-forms of a certain number of copies of groups of type $G_2$. Recall that a polarization is called $p$-principal if its degree is 
coprime to $p$. Our construction hinges on a choice of a CM field of degree $2\cdot8=16$, in fact an elaboration of the method of \cite{habil} yields the following more specific result.

\begin{thm}
\label{G2}
Suppose that $L^+$ is a totally real number field of degree $r>7$. Assume that some odd rational prime $p$ is inert and unramified in $L^+$, so that 
$\q_p\otimes L^+\cong L_{\gq^+}^+\cong K(\f_{p^r})$, where $\gq^+$ is the sole prime of $L^+$ over $p$. Moreover, let $L$ be a totally imaginary quadratic extension of $L^+$ 
and assume that $\gq^+$ splits in $L$, so that $\gq^+=\gq\gq^*$, where $*$ denotes the non-trivial element of $\Gal(L/L^+)$ and $\gq$ is one of the two primes of $L$ over $\gq^+$. 
Then there exists a $p$-principally polarized Abelian $7r=g$-fold with Rosati invariant $\O_L$-action $(Y,\lambda,\iota)$ over some finitely generated extension $F$ of $\f_p$ such that: 
\begin{itemize}
\item
For every prime $\gr\nmid p$ of $L$, the smallest $L_\gr$-algebraic subgroup of $\GL(V_\gr/L_\gr)$ containing $\rho_\gr(\Gal(\Fbar/F))$ agrees with the product of the 
homotheties with a simple group of type $G_2$ over $L_\gr$, where $\rho_\gr$ denotes the natural $\Gal(\Fbar/F)$-action on $V_\gr=\q\otimes\invlim_n Y(\Fbar)[\gr^n]$.
\item
The formal isogeny type of $Y[\gq^\infty]$ is $G_{0,1}^{2r}\oplus G_{1,r-1}^3\oplus G_{2,r-2}^2$.
\item
The ghost (in the sense of \cite[Section 6]{cadoret}) of $Y$ is an Abelian $r$-fold allowing complex 
multiplication by $L$ and the formal isogeny type of its $p$-divisible group is $G_{1,r-1}\oplus G_{r-1,1}$.
\end{itemize}
\end{thm}

The proof of theorem \ref{G2} is explained in the subsection \ref{proof}. Our assumption "$r>7$" enters into a construction aiming at a description of the $p$-divisible group $Y[\gq^{*\infty}]$ 
with $W(\f_{p^r})$-action over $k[[t]]$. The idea is to choose $Y[\gq^{*\infty}]$ in the isogeny class of a direct sum of a constant $p$-divisible group of height $3r$ with two copies of a 
non-constant $p$-divisible group of height $2r$ while the dimensions of the $\f_{p^r}$-eigenspaces of $\Lie Y[\gq^{*\infty}]$ are as big as possible, so that the Newton slopes of the generic 
fiber are the ones given in theorem \ref{G2}. This construction is explained in the next subsection, which is an elaboration of \cite[Subsection 2.2]{habil}. With a little bit of extra work theorem 
\ref{G2} can probably be proved for any $r\geq4$, possibly by using the improved method of \cite{E7}. It is tempting to speculate on the cases $r\in\{2,3\}$, which could be consequences 
of Matthew Emerton's $p$-adic variational Hodge conjecture, cf. \cite[Conjecture(2.2)]{emerton}. For the case $r=8$, our construction was announced in the introduction of \cite{habil}.

\subsection{On $\z/r\z$-graded Frobenius modules with $\SL(2)_{K(\f_{p^r})}$-structure}

We need to introduce Zink's windows in the generality which we are going to use, namely over $k[[t]]$, where $k$ is a perfect field of characteristic $p$. Let us write $\tau$ for the Frobenius lift 
on $W(k)[[t]]$ with $\tau(t)=t^p$. A Dieudonn\'e $W(k)[[t]]$-window is a triple $(M,M_1,\phi)$, where $M$ is a finitely generated free $W(k)[[t]]$-module, $M_1\subset M$ is a $W(k)[[t]]$-submodule 
such that $M/M_1$ is a free $k[[t]]$-module and $\phi:M\rightarrow M$ is a $\tau$-linear homomorphism such that $\phi(M_1)$ generates the $W(k)[[t]]$-submodule $pM$. Zink's nilpotence condition 
\cite[Definition 3]{zink1} defines his full subcategory of $W(k)[[t]]$-windows, which turns out to be equivalent to the category of formal $p$-divisible groups over $k[[t]]$, according to \cite[Theorem 4]{zink1}. 
We will write $\BT$ for the equivalence from the former to the latter. The transition from windows to nilpotent displays is achieved with the observation that the image of the (injective) ghost morphism
\begin{eqnarray*}
&W(W(k)[[t]])\hookrightarrow W(k)[[t]]^{\n_0};\,(x_0,x_1,\dots)\mapsto(w_0,w_1,\dots)&\\
&w_n:=\sum_{i=0}^np^ix_i^{p^{n-i}}& 
\end{eqnarray*}
is the subring $\{(w_0,w_1,\dots)\in W(k)[[t]]^{\n_0}\mid\,w_i\equiv\tau(w_{i-1})\pmod{p^i}\forall i\in\n\}$, giving rise to a homomorphism $\kappa:W(k)[[t]]\rightarrow W(W(k)[[t]])$ which satisfies 
$w_i\circ\kappa=\tau^i$ for all $i\in\n_0$ and is called Cartier's diagonal homomorphism. If $\bark$ denotes the precomposition of $\kappa$ with the natural reduction $W(W(k)[[t]])\rightarrow W(k[[t]])$, 
then the display theoretic pendant $(P,Q,F,V^{-1})$ of a window $(M,M_1,\phi)$ is obtained by taking $P=W(k[[t]])\otimes_{\bark,W(k)[[t]]}M$ while $Q$ is the inverse image of the 
$k[[t]]$-module $M/M_1$. The $F$-linear operators $F$ and $V^{-1}$ are induced from the $\tau$-linear operator $\phi$ together with $F\circ\bark=\bark\circ\tau$. Working over 
the smaller ring $W(k)[[t]]\hookrightarrow W(k[[t]])$ paves the way for analyzing the monodromy properties of a $p$-divisible group over $k[[t]]$ by using the faithfully flat extension 
$$W(k)[[t]][\frac1p]\subset K(k)\{\{t\}\}:=\{\sum_ia_it^i|a_i\in K(k),\,v_p(a_i)+i\epsilon\to\infty\,\forall\epsilon>0\}$$
for a trivialization of the Frobenius, which is due to Bernard Dwork. By a $\z/r\z$-gradation on a $W(k)[[t]]$-window 
$(M,M_1,\phi)$ we mean compatible $\z/r\z$-gradations on $M$ and $M_1$ such that $\phi$ is homogeneous of degree $1$.

\begin{lem}
\label{twoSL}
Fix $r>7$ and an auxiliary $\z/r\z$-graded Dieudonn\'e module $H=\bigoplus_{\sigma=1}^rH_\sigma$ of formal isogeny type $G_{r-1,1}$. Then there exists 
a $\z/r\z$-graded $W(k)[[t]]$-window $\tilde I=\bigoplus_{\sigma=1}^r\tilde I_\sigma$ whose special (resp. generic) fibre is of formal isogeny type $G_{r-1,1}^2$ 
(resp. $G_{r-2,2}\oplus G_{1,0}^r$) and a $\z/r\z$-graded $W(k)[[t]]$-window $\tilde M=\bigoplus_{\sigma=1}^r\tilde M_\sigma$ which is isogenous to 
$$\tilde I^{\oplus2}\oplus W(k)[[t]]\otimes_{W(k)}H^{\oplus3}$$ 
and satisfies $\rk_{k[[t]]}\tilde M_\sigma/\tilde M_{\sigma,1}\geq6$ for every $\sigma$.
\end{lem}
\begin{proof}
The assumption on the formal isogeny type of $H=\bigoplus_{\sigma=1}^rH_\sigma$ implies that $\rk_{W(k)}H_\sigma=1$ 
(resp. $\dim_kH_\sigma/H_{\sigma,1}=1$) holds for every (resp. for all but one) element $\sigma$ of $\z/r\z$, i.e.
\begin{eqnarray*}
&&\dim_kH_\sigma/H_{\sigma,1}=\begin{cases}0&\sigma\equiv\sigma_1\pmod r\\1&\mbox{ otherwise}\end{cases}\\
&&\rk_{W(k)}H_\sigma=1
\end{eqnarray*}
for some $\sigma_1\in\z$. Notice that the requested properties of $\tilde I$ (resp. $\tilde M$) force its special fiber to lie in the 
isogeny class of $H^{\oplus2}$ (resp. $H^{\oplus7}$). Fix $\sigma_2\in\z$ satisfying $5\leq\sigma_2-\sigma_1\leq r-2$, along with 
a $\z/r\z$-graded Dieudonn\'e module $I=\bigoplus_{\sigma=1}^rI_\sigma$ of formal isogeny type $G_{r-1,1}^2$ and satisfying
\begin{eqnarray*}
&&\dim_kI_\sigma/I_{\sigma,1}=\begin{cases}1&\sigma\equiv\sigma_1\pmod r\\1&\sigma\equiv\sigma_2\pmod r\\2&\mbox{ otherwise}\end{cases}\\
&&\rk_{W(k)}I_\sigma=2
\end{eqnarray*}
for every $\sigma$. Working in the category of windows, we describe an equicharacteristic deformation of $I$ whose generic fiber has the 
formal isogeny type $G_{r-2,2}\oplus G_{1,0}^r$. We start out from $\tilde I_\sigma:=W(k)[[t]]\otimes_{W(k)}I_\sigma$ and define a new 
Frobenius thereon by precomposition (of the $\tau$-linear extension to $\tilde I_{\sigma-1}$) of $\phi:I_{\sigma-1}\rightarrow I_\sigma$ with 
$$U_\sigma:=\begin{cases}\id_{\tilde I_\sigma}+t\otimes N_1&\sigma\equiv\sigma_1\pmod r\\
\id_{\tilde I_\sigma}+t\otimes N_2&\sigma\equiv\sigma_2\pmod r\\
\id_{\tilde I_\sigma}&\mbox{ otherwise}\end{cases}$$
where $N_j$ denotes endomorphisms of $I_{\sigma_j}$ satisfying $\ker(N_j)=N_jI_{\sigma_j}\nsubseteq I_{{\sigma_j},1}$ for $j\in\{1,2\}$. Indeed, it is known that 
$W(k((t)))\otimes_{W(k)[[t]]}\tilde I$ has non-zero $p$-rank, at least for good choices of $N_1$ and $N_2$ according to \cite[Proposition 4.1.4]{wedhorn}. The definition of the window 
$\tilde I$ is not completed before one has decreed $\tilde I_{\sigma,1}:=p\tilde I_\sigma+W(k)[[t]]\otimes_{W(k)}I_{\sigma,1}$. Our prime interest lies in $\z/r\z$-graded Dieudonn\'e sublattices: 
\begin{equation}
\label{comp}
M_\sigma\subset I_\sigma^{2\oplus}\oplus H_\sigma^{3\oplus}
\end{equation}
We require that $M_\sigma$ satisfies $\dim_kM_\sigma/M_{\sigma,1}\geq6$ for every $\sigma$ and that \eqref{comp} is an equality for $\sigma\in\{\sigma_1,\sigma_2\}$. 
Let us check that lattices with these properties exist: Starting out from $M_j:=I_{\sigma_j}^{\oplus2}\oplus H_{\sigma_j}^{\oplus3}$ we observe that the $W(k)$-length 
of $M_2/\phi^{\sigma_2-\sigma_1}(M_1)$ (resp. $M_1/\phi^{\sigma_1-\sigma_2+r}(M_2)$) is equal to $5$ (resp. equal to $2$). So let us pick flags of $W(k)$-modules 
$$pM_2\subsetneq\phi^{\sigma_2-\sigma_1}(M_1)=F_0\subsetneq F_1\subsetneq F_2\subsetneq F_3\subsetneq F_4\subsetneq M_2=F_5=\dots$$
(resp. $pM_1\subsetneq\phi^{\sigma_2-\sigma_1+r}(M_2)=E_0\subsetneq E_1\subsetneq M_1=E_2=\dots$) and define $M_\sigma=\phi^{\sigma-\sigma_2}(F_{\sigma-\sigma_1})$ 
provided that $\sigma_1\leq\sigma\leq\sigma_2$ (resp. $M_\sigma=\phi^{\sigma-\sigma_1}(E_{\sigma-\sigma_2+r})$ whenever $\sigma_2-r\leq\sigma\leq\sigma_1$). We still have to 
construct our $\z/r\z$-graded $W(k)[[t]]$-window $\bigoplus_{\sigma=1}^r\tilde M_\sigma$. Again, we start out from $\tilde M_\sigma:=W(k)[[t]]\otimes_{W(k)}M_\sigma$ 
and define a new Frobenius thereon by precomposition (of the $\tau$-linear extension to $\tilde M_{\sigma-1}$) of $\phi:M_{\sigma-1}\rightarrow M_\sigma$ with 
$$U_\sigma:=\begin{cases}\id_{\tilde M_\sigma}+t\otimes N'_1&\sigma\equiv\sigma_1\pmod r\\
\id_{\tilde M_\sigma}+t\otimes N'_2&\sigma\equiv\sigma_2\pmod r\\
\id_{\tilde M_\sigma}&\mbox{ otherwise}\end{cases}$$
where $N'_j$ denotes the endomorphism of $M_j$ which agrees with $N_j$ on the two copies of $I_{\sigma_j}$ and vanishes on the three copies of $H_{\sigma_j}$.
\end{proof}

We use the terminology "$F$-isocrystal" for pairs $(M,\phi)$ consisting of a finite dimensional $K(k)$-vector space $M$ and an isomorphism $^FM:=K(k)\otimes_{F,K(k)}M\stackrel{\phi}{\rightarrow}M$ 
while "Frobenius-module" is used for pairs $(M,\phi)$ consisting of a finitely generated free $W(k)[[t]][\frac1p]$-module $M$ and an isomorphism
$$^\tau M:=W(k)[[t]][\frac1p]\otimes_{\tau,W(k)[[t]][\frac1p]}M\stackrel{\phi}{\rightarrow}M,$$ 
where $\tau$ was defined at beginning of this subsection. By a $\z/r\z$-gradation on an $F$-isocrystal or Frobenius-module $(M,\phi)$ we mean a decomposition $M=\bigoplus_{\sigma=1}^rM_\sigma$ 
satisfying $\phi(M_\sigma)\subset M_{\sigma+1}$. The category of $\z/r\z$-graded Frobenius-modules forms a $K(\f_{p^r})$-linear rigid $\otimes$-category in the usual way. Moreover, there exist 
two interesting $\otimes$-functors to the $K(\f_{p^r})$-linear rigid $\otimes$-categories of $\z/r\z$-graded $F$-isocrystals, namely the formation of the special fiber over $k$, i.e. $M\mapsto M/tM$ 
and the formation of the generic fiber over the perfection $k((t))^{perf}$ of $k((t))$, i.e. $M\mapsto K(k((t))^{perf})\otimes_{W(k)[[t]][\frac1p]}M$. The latter employs Cartier's diagonal morphism 
$W(k)[[t]]\rightarrow W(W(k)[[t]])$ precomposed with $W(W(k)[[t]])\rightarrow W(k((t))^{perf})$. If $k$ is algebraically closed, then the $K(\f_{p^r})$-linear rigid full $\otimes$-subcategory of $\z/r\z$-graded 
$F$-isocrystals whose Newton slopes are zero is equivalent to the category of finite dimensional $K(\f_{p^r})$-vector spaces by means of the skeleton, which is the functor $M\mapsto\{x\in M_0\mid \phi^r(x)=x\}$. 
In the following result $\bRep_0(\SL(2)_{K(\f_{p^r})})$ stands for the $K(\f_{p^r})$-linear tannakian category of finite dimensional representations of the group $\SL(2)$ over the ground field $K(\f_{p^r})$.

\begin{lem}
\label{oneSL}
Consider a $\z/r\z$-graded Frobenius-module $I=\bigoplus_{\sigma=1}^rI_\sigma$ over an algebraically closed ground field $k$ of characteristic $p$. Assume that all Newton 
slopes of its special (resp. generic) fiberare zero (resp. non-zero) and that each $I_\sigma$ is free of rank $2$. Then there exists a fully faithful $K(\f_{p^r})$-linear rigid 
$\otimes$-functor $M$ from $\bRep_0(\SL(2)_{K(\f_{p^r})})$ to the $K(\f_{p^r})$-linear rigid $\otimes$-category of $\z/r\z$-graded Frobenius-modules such that the following hold:
\begin{itemize}
\item[(i)]
The canonical fiber functor on $\bRep_0(\SL(2)_{K(\f_{p^r})})$ is isomorphic to the $K(\f_{p^r})$-linear rigid 
$\otimes$-functor $\rho\mapsto S(\rho)$ where $S(\rho)$ denotes the skeleton of the special fiber of $M(\rho)$.
\item[(ii)]
Applying $M$ to the standard representation of $\SL(2)_{K(\f_{p^r})}$ yields $I$.
\end{itemize}
\end{lem}
\begin{proof}
Let $J$ be the skeleton of the special fiber of $I$, clearly $\dim_{K(\f_{p^r})}J=2$. Using the methods of \cite[Proposition 29]{meszin} there is a canonical $\phi^r$-equivariant isomorphism
\begin{equation}
\label{zeroSL}
\Theta:I_0\otimes_{W(k)[[t]][\frac1p]}K(k)\{\{t\}\}\stackrel{\cong}{\rightarrow}J\otimes_{K(\f_{p^r})}K(k)\{\{t\}\}
\end{equation}
where $K(k)\{\{t\}\}\subset K(k)[[t]]$ denotes the subring of power series that converge on the open unit 
disc. Let $G_0$ be the smallest $K(\f_{p^r})$-subgroup of $\GL(J/K(\f_{p^r}))$ containing the element:
$$\theta:=(\Theta\otimes1)\circ(1\otimes\Theta)^{-1}\in\GL(J/K(\f_{p^r}))(K(k)\{\{t\}\}\otimes_{W(k)[[t]]}K(k)\{\{t\}\})$$
Since $W(k)[[t]][\frac1p]\rightarrow K(k)\{\{t\}\}$ is faithfully flat we can use descent theory to construct a fully faithful functor $M$ from $\bRep_0(G_0)$ to the $K(\f_{p^r})$-linear rigid $\otimes$-category 
of $\z/r\z$-graded Frobenius-modules. Notice $G_0\subset\SL(J/K(\f_{p^r}))$, because $\bigoplus_{\sigma=1}^r\bigwedge_{W(k)[[t]][\frac1p]}^2I_\sigma$ must be constant. It remains to prove that $G_0$ 
contains $\SL(J/K(\f_{p^r}))$. The assumption on the slopes of the generic fiber leads to $0<\dim_{K(\f_{p^r})}G_0$. However, the only maximal proper subgroups of $\SL(J/K(\f_{p^r}))$ of positive dimension 
are the Borel group and (four different forms of) $\g_m\rtimes\{\pm1\}$. Suppose $G_0\cong\g_m\rtimes\{\pm1\}$ for instance, so that there is a decomposition $J=J'\oplus J''$ such that $G_0$ stabilizes 
$J'\cup J''$. Let us write $g_0$ for the unique non-trivial involution of $\PGL(J/K(\f_{p^r}))$ whose eigenspaces are $J'$ and $J''$ with eigenvalues $\pm1$, and observe that $g_0G_0g_0^{-1}=G_0$. Descent 
theory yields a global section $s_0$ of $\PGL(I_0/W(k)[[t]][\frac1p])$ such that $\Theta s_0\Theta^{-1}$ agrees with $g_0$, simply because $(\Theta\otimes1)^{-1}g_0(\Theta\otimes1)=(1\otimes\Theta)^{-1}g_0(1\otimes\Theta)$ 
is implied by $g_0\theta g_0^{-1}=\theta$. Using that $W(k)[[t]][\frac1p]$ is a principal ideal domain we can find a global section 
$s$ of $\GL(I_0/W(k)[[t]][\frac1p])$ which lifts $s_0$. Moreover, it does no harm to assume that the $p$-adic valuation of $s^2=:a\in K(k)^\times=W(k)[[t]][\frac1p]^\times$ is $0$ or $1$. The $\phi^r$-equivariance 
of \eqref{zeroSL} implies that $s$ commutes with the $r$th iterate of the Frobenius on $I_0$. Consequently this object can be regarded as a $W(k)[[t]][\frac1p,x]/(x^2-a)$-module of rank one together with 
an isomorphism $^{\tau^r}I_0\stackrel{\phi^r}{\rightarrow}I_0$, which therefore does not allow non-trivial specializations of Newton-polygons. We leave to the reader to check that in each of the other three cases the 
special and generic Newton-polygons of a Frobenius module with $G_0$-structure would agree too, which stands in contradiction to the assumptions on the Newton-polygons of $I=\bigoplus_{\sigma=1}^rI_\sigma$.
\end{proof}

\subsection{Proof of theorem \ref{G2}}
\label{proof}
Consider the $\q$-group $G:=\Auto(\8)$, where $\8$ stands for the $8$-dimensional division algebra of octonions over $\q$. The Lie-algebra $\gg$ of $G$ can be identified with the space of 
derivations of $\8$ and the actions of both $G$ and $\gg$ preserve the positive definite symmetric form $\8\times\8\rightarrow\q;(x,y)\mapsto x\ybar+y\xbar$ and the $7$-dimensional subspace 
$\8_0=\{x\in\8\mid\,\xbar=-x\}$, where $\8\rightarrow\8;x\mapsto\xbar$ is the canonical conjugation. Notice that $G(\r)$ is compact because it preserves a positive definite form. We need a group theoretical lemma:

\begin{lem}
If $G$ is as above, then there exists a homomorphism
\begin{equation}
\label{root}
\pi:\SL(2)_{K(\f_{p^r})}\rightarrow G_{K(\f_{p^r})}
\end{equation}
satisfying the following:
\begin{itemize}
\item[(i)]
No proper $L^+$-subgroup of $G_{L^+}$ contains the image of $\pi$.
\item[(ii)]
There exists an isomorphism $K(\f_{p^r})\otimes\8_0\cong\sd_{K(\f_{p^r})}^{\oplus2}\oplus K(\f_{p^r})^{\oplus3}$ in the category of $\SL(2)_{K(\f_{p^r})}$-representations. 
\end{itemize}
\end{lem}
\begin{proof} 
Notice that $G_{\q_p}$ is split, for example by \cite{springer}. Homomorphisms satisfying (ii) arise from the long simple root in some based root system of the split group $G_{\q_p}$, 
alternatively one could think of our embedded $\SL(2)$ as the commutator subgroup of the Levi factor of the maximal proper standard parabolic subgroup arising from the 
removal of the short simple root. The lemma follows because the set of homomorphisms satisfying (ii) and violating (i) is a countable union of nowhere dense closed subsets.
\end{proof}

Let $\Xbar$ be a polarized Abelian variety with complex multiplication by $\O_L$ such that the formal isogeny type of $\Xbar[\gq^{*\infty}]$ is $G_{r-1,1}$. Following \cite[Lemme 5]{tate} it can be 
constructed from a suitable CM-type. As $\gq$ and $\gq^*$ are the only prime divisors of $p\O_L$, we can write the complex embeddings of $L$ as $\iota_0\circ F^i\circ\iota_\gq\circ*^j$, where:
\begin{itemize}
\item
$\iota_0:L_\gq\rightarrow\C$ is a fixed embedding.
\item
$\iota_\gq:L\rightarrow L_\gq;\alpha\mapsto\alpha_\gq$ denotes the passage to the $\gq$-adic completion.
\item
$F$ denotes the absolute Frobenius on $L_\gq\cong K(\O_L/\gq)$.
\item
$i$ runs through $\{0,\dots,r-1\}$.
\item
$j$ runs through $\{0,1\}$.
\end{itemize}
The CM type we wish to pick is $\Phib:=\{\iota_0\circ\iota_\gq,\iota_0\circ F\circ\iota_\gq\circ*,\dots,\iota_0\circ F^{r-1}\circ\iota_\gq\circ*\}$, so that a resulting CM-Abelian variety is the cokernel of 
\begin{equation}
\label{converse1}
\O_L\rightarrow\C^r;\alpha\mapsto(\iota_0(\alpha_\gq),\iota_0({^F(\alpha^*)_\gq}),\dots,\iota_0({^{F^{r-1}}(\alpha^*)_\gq})),
\end{equation}
which is definable over the integral closure $\O$ of $\iota_0(L_\gq)$ in $\C$. We obtain $\Xbar$ as its $\mod\gP$-reduction, where $\gP$ is the maximal ideal of $\O$, which can be identified with the ring of integers of $\qbar_p$. 
It goes without saying that its ground field is $k:=\O/\gP$, i.e. the algebraic closure of $\O_L/\gq$. We also need to turn \eqref{converse1} and $\Xbar$ into polarized CM Abelian varieties by choosing a polarization of the form 
\begin{equation}
\label{converse2}
\O_L^2\ni(x,y)\mapsto2i\pi\tr_{L/\q}(vxy^*),
\end{equation}
for an auxiliary element $-v^*=v\in\O_L\backslash\{0\}$ such that $\iota_0({^Fv_\gq})$, ..., $\iota_0({^{F^{r-1}}v_\gq})$ are lying in the upper half plane and $\iota_0(v_\gq)$ in the lower one. Consider the $\z/r\z$-graded windows 
$\tilde I$ and $\tilde M$, as provided by lemma \ref{twoSL}, when applied to the $\z/r\z$-graded Dieudonn\'e module $H$ with $\BT(\bigoplus_{\sigma=1}^rH_\sigma)\cong\Xbar[\gq^{*\infty}]$. Moreover, let us write $M$ for 
the fully faithful $\z/r\z$-graded Frobenius module with $\SL(2)_{K(\f_{p^r})}$-structure resulting from applying lemma \ref{oneSL} to $\q\otimes\bigoplus_{\sigma=1}^r\Hom_{W(k)}(H_\sigma,\tilde I_\sigma)$. Observe that the 
special fiber $\BT(\tilde M/t\tilde M)$ of $\BT(\tilde M)$ is canonically isogenous to $\8_0\otimes\Xbar[\gq^{*\infty}]$, simply because part (ii) of lemma \ref{oneSL} tells us that $\q\otimes\tilde M$ agrees with $M(\pi)\otimes H$ 
(by slight abuse of notation we may regard $\pi$ as a representation of $\SL(2)_{K(\f_{p^r})}$ on $L_{\gq^*}\otimes\8_0\cong\sd_{L_{\gq^*}}^{\oplus2}\oplus L_{\gq^*}^{\oplus3}$). This puts us into a position allowing the use of the 
Serre-Tate theorem: Over $k[[t]]$ there exists a canonical $p$-principally polarized Abelian scheme $Y^{(1)}$ with Rosati-invariant $\O_L$-operation such that its special fiber lies in the isogeny class $\8_0\otimes\Xbar$ while:
\begin{equation}
\label{Opa01}
Y^{(1)}[\gq^{*\infty}]\cong\BT(\bigoplus_{\sigma=1}^r\tilde M_\sigma)
\end{equation}
This implies that the $k[[t]]$-ranks of the $F^\sigma\circ\iota_\gq$-eigenspaces of $\Lie Y^{(1)}$ are at most equal to one and we let $\Omega$ consist of all $\sigma\in\{1,\dots,r\}$ for which the said eigenspace 
is of $k[[t]]$-rank equal to one (N.B.: $\Omega=\{\sigma_1,\sigma_1+1,\sigma_1+2,\sigma_1+3,\sigma_1+4,\sigma_2,\sigma_2+1\}$ if $\sigma_1$ and $\sigma_2$ are as in the proof of lemma \ref{twoSL}).\\ 
The crux of our argument is the 2nd exterior power Abelian scheme, which was discovered somewhat implicitly in \cite[Chapter IV, Paragraph 5, Exercise 1]{satake} over $\C$ and was generalized and reconsidered in 
\cite[subsection 4.3]{habil}. Its construction necessitates the introduction of integral models over $W(\O_L/\gq)$ of certain Shimura varieties of PEL type. Their moduli interpretations involve $p$-principally polarized products 
$Y^{(0)}\times_SY^{(1)}$ of two Abelian schemes with Rosati invariant $\O_L$-operations over variable bases $S\rightarrow\Spec W(\O_L/\gq)$ such that $\dim_SY^{(0)}=r$, all $F^\sigma\circ\iota_\gq$-eigenspaces 
of $\Lie Y^{(0)}$ vanish, $\dim_SY^{(1)}=7r$ and each $F^\sigma\circ\iota_\gq$-eigenspace of $\Lie Y^{(1)}$ is an invertible $\O_S$-module (resp. vanishes) whenever $\sigma\in\Omega$ (resp. 
$\sigma\notin\Omega$), supplemented with some bookkeeping of level structures, following \cite{kottwitz}. The resulting moduli scheme, called $\M_{\gp}^{(0\times1)}$ in loc.cit., serves as a base for a universal 
$p$-principally polarized 2nd exterior power Abelian scheme $Y^{(2)}$, which is equipped with a Rosati-invariant $\O_L$-operation, is of relative dimension $21r$ and has $F^\sigma\circ\iota_\gq$-eigenspaces 
of $\Lie Y^{(2)}$ that are a locally free $\O_{\M_\gp^{(0\times1)}}$-module of rank $6$ or $0$ depending on whether or not $\sigma\in\Omega$. In fact, there is a canonical isomorphism of polarized Hodge structures 
\begin{equation}
\label{bla}
TY_\xi^{(0)}\otimes_{\O_L}TY_\xi^{(2)}\cong\bigwedge_{\O_L}^2TY_\xi^{(1)}
\end{equation}
for every $\xi:\Spec\C\rightarrow\M_\gp^{(0\times1)}$. The restriction of $Y^{(2)}$ to the generic fiber of $\M_\gp^{(0\times1)}$ arises from combining functoriality properties of 
canonical models with their moduli interpretations (cf. \cite[Corollaire 5.4]{deligne4}, \cite[Th\'eor\`eme 4.21]{deligne4}). Please see to \cite[subsection 4.3]{habil} and its references for 
explanations of the extension process to the whole of $\M_\gp^{(0\times1)}$. It turns out that $\M_\gp^{(0\times1)}$ is not only smooth, but also projective, thanks to $\sharp\Omega=7<r$. 
If one writes $N$ for the $\z/r\z$-graded Dieudonn\'e module with $\BT(\bigoplus_{\sigma=1}^rN_\sigma)\cong Y^{(0)}[\gq^{*\infty}]$, then the 2nd exterior power $Y^{(2)}$ satisfies:
\begin{equation}
\label{Opa02}
Y^{(2)}[\gq^{*\infty}]\cong\BT(\bigoplus_{\sigma=1}^r\Hom_{W(k)}(N_\sigma,\bigwedge_{W(k)}^2\tilde M_\sigma))
\end{equation}
Moreover, according to \cite[Proposition 5.1]{habil} we have a commutative diagram
\begin{equation}
\label{Opa03}
\begin{CD}
\sy_L^2\End_L^0(Y^{(1)}\times_{k[[t]]}k)@>>>\sy_{K(\f_{p^r})}^2\End(\bigoplus_{\sigma=1}^r\tilde M_\sigma/t\tilde M_\sigma)\\
@VVV@VVV\\
\End_L^0(Y^{(2)}\times_{k[[t]]}k)@>>>\End(\bigoplus_{\sigma=1}^r\bigwedge_{W(k)}^2\tilde M_\sigma/t\tilde M_\sigma)
\end{CD},
\end{equation}
where the horizontal arrows are induced from the isomorphisms \eqref{Opa01} and \eqref{Opa02}. We have a decomposition $\bigwedge^2\8_0\cong\gg\oplus\8_0$, of which the projection to the first 
summand is sketched in \cite[Chapter V, Section 19, Exercise 5]{humphreys} while its projection to the second summand results from the commutator of octonions. The full faithfulness of $M$ implies
\begin{eqnarray*}
&&\End^0_L(Y^{(1)})=\{\alpha\in L\otimes \End(\8_0)\mid\,\{\alpha_{\gq^*},\alpha_{\gq^*}^*\}\subset\End_{\SL(2)}(L_{\gq^*}\otimes\8_0)\}=\\
&&L\otimes\End_G(\8_0)=L\\
&&\End^0_L(Y^{(2)})=\{\alpha\in L\otimes\End(\bigwedge^2\8_0)\mid\,\{\alpha_{\gq^*},\alpha_{\gq^*}^*\}\subset\End_{\SL(2)}(L_{\gq^*}\otimes\bigwedge^2\8_0)\}\\
&&=L\otimes\End_G(\bigwedge^2\8_0)\cong L\oplus L,
\end{eqnarray*}
where $\alpha_{\gq^*}$ (resp. $\alpha_{\gq^*}^*$) denotes the image of $\alpha$ (resp. $\alpha^*$) in $L_{\gq^*}\otimes\End(\8_0\mbox{ or }\bigwedge^2\8_0)$. The final step of the proof consists of choosing a model 
$Y$ of the generic fiber of $Y^{(1)}$ over some finitely generated subfield $F\subset k((t))$. It does no harm to assume that all endomorphisms of $Y$ and its 2nd exterior power are defined over $F$, and we also decree 
$F$ to contain $\f_{p^r}$. The result follows from applying Zarhin's theorem to $Y$ and its 2nd exterior power, combined with some multilinear bookkeeping of Tate modules involving $\ell$-adic analogs of \eqref{Opa02} and \eqref{Opa03}.\\
The ghost $\gY$ of $Y$ must be a power of $\Xbar$, since $\Xbar^{\oplus7}$ is isogenous to a specialization of $Y$ while 
$\Xbar$ is simple, because the formal isogeny type $G_{1,r-1}\oplus G_{r-1,1}$ cannot be written as a sum of two self-dual ones. In order to show the last assertion of theorem \ref{G2}, 
it remains to show that $[L:\q]=\dim_{\q_\ell}V_\ell Y_{\Fbar}^{T_\ell}$ for some maximal torus $T_\ell\subset G_\ell^{der}$. Indeed, observe that $G_\ell$ commutes with $L$ so that: 
$$V_\ell Y_{\Fbar}^{T_\ell}\cong\bigoplus_{\gr\mid\ell}V_\gr Y_{\Fbar}^{T_\ell}$$
Moreover, the description of the Zariski closure of $\rho_\gr(\Gal(\Fbar/F))$ shows that $V_\gr Y_{\Fbar}^{T_\ell}$ 
is a one-dimensional vector space over $L_\gr$ for each $\gr$ and $[L:\q]=\sum_{\gr\mid\ell}[L_\gr:\q_\ell]$.

\begin{rem}
Please see to \cite{H1} for an explanation of exterior powers of one-dimensional $p$-divisible groups by means of a multilinear Dieudonn\'e theory, as suggested 
by Richard Pink and Hadi Hedayatzadeh. Eventually, this theory has lead to a $(\invlim_n Y^{(0)}[\gq^n])\otimes_{\O_{L_\gq}}Y^{(2)}[\gq^\infty]$-valued 
alternating pairing on  $Y^{(1)}[\gq^\infty]$, please see to \cite[Construction 2.5]{H2} for more general assertions.
\end{rem}

\section{On two Moduli Spaces}
\label{converse}
Our two examples arose from $\fbar_p[[t]]$-sections in moduli spaces of Abelian varieties with a certain kind of additional structure. We round off the 
treatment with soberly introducing these moduli spaces, whose ties to the theory of Shimura varieties deserve further study, as initiated in \cite{E7}.

\subsection{First moduli space}
Recall that over an arbitrary number field, isometry classes of three-dimensional quadratic spaces with discriminant $1$ are classified by the sets of their anisotropic places, which are arbitrary finite sets of even cardinality. 
Specializing to our totally real quadratic field $K$, we fix an embedding $v:K\hookrightarrow\r$ and a quadratic space $V$ which is isotropic at $v$ and anisotropic at the other real embedding of $K$. Consider an odd 
rational prime $p$ which is inert and unramified in $K$ and such that $V$ is isotropic at the unique prime above $p$. Notice that the kernel of the diagonal $\z_{(p)}\otimes\O_K\otimes\O_K\rightarrow\z_{(p)}\otimes\O_K$ is 
generated by a unique idempotent, which we denote by $e$. Let $\O_D\supset\O_K$ be as in subsection \ref{morita1} and let $(\gY,\lambda_\gY)$ be a polarized Abelian surface with a Rosati invariant $\O_D$-action over 
$\O_K/p\O_K$, say with $\deg(\lambda_\gY)=p^2$ and $\Lie\gY[1-e]=0$ to fix ideas. Mimicking the formalism of \cite[Section 5]{kottwitz} we introduce the locally compact rings of adeles $\a:=\r\times\q\otimes\hat\z$ and 
$\a^{\infty,p}:=\q\otimes\prod_{\ell\neq p}\z_\ell$ and consider the $\O_K/p\O_K$-functor $\Mbar$ whose value on some connected $\O_K/p\O_K$-scheme $S$ is given by the set of quadruples $(Y,\lambda,\iota,\eta)$ with the 
following properties:

\begin{itemize}
\item
$Y\rightarrow S$ is an Abelian $6$-fold, equipped with an action $\iota:\O_D\rightarrow\End(Y)$, up to $\z_{(p)}$-isogeny. Moreover, we require that $\Lie Y[1-e]$ (resp. $\Lie Y[e]$) 
is a projective $\O_S$-module of rank $2$ (resp. $4$), here notice that $e$ gives rise to an idempotent in $\O_K\otimes\O_S$, so that $\Lie Y=\Lie Y[1-e]\oplus\Lie Y[e]$.
\item
$\lambda:Y\dashrightarrow Y^t$ is a $p$-integral quasipolarization (coming from a positive element in the Neron-Severi group tensorized with $\z_{(p)}$) which satisfies 
$\lambda\circ\iota(\alpha^*)=\iota(\alpha)^t\circ\lambda$ for any $\alpha\in\O_D$. We require that the induced isogeny $Y[p^\infty]\rightarrow Y^t[p^\infty]$ is of 
degree $p^2$. (N.B.: This implies that its kernel is contained in $Y[p]$, by \eqref{O2} and \cite[Preliminaries]{oort4}.)
\item
$\eta:V\otimes_K H_1^{\mathaccent19 et}(\gY_s,\a^{\infty,p})\stackrel{\cong}{\rightarrow}H_1^{\mathaccent19 et}(Y_s,\a^{\infty,p})$ is a $\pi_1^{\mathaccent19 et}(S,s)$-invariant 
$\a^{\infty,p}\otimes D$-linear isometry, where $s$ is an arbitrary geometric point on $S$ (N.B.: Both sides possess natural $\a^{\infty,p}(1)$-valued pairings).
\end{itemize}

Let $\GO(V/K)=:H$ be the $K$-group of $K$-linear isometries of $V$. Observe that every $g\in H(\a^{\infty,p}\otimes K)$ gives rise to an automorphism of $\Mbar$, as $(Y,\lambda,\iota,\eta)$ 
can be sent to $(Y,\lambda,\iota,\eta\circ g)$. One may speculate on whether or not $\Mbar$ possesses analogues in the theory of Rapoport-Zink spaces of PEL type in the sense of 
\cite[Definition 3.21]{rapoport}, but it seems hard to apply loc.cit. directly. This is due to condition (iii) of \cite[Definition 3.18]{rapoport}, which requires our associated reductive $\q_p$-group 
to possess a cocharacter with weights $0$ and $1$ in its standard representation, thus ruling out the orthogonal group in three variables. Nevertheless, it seems worthwhile to try 
to adapt \cite{rapoport} to the case at hand and similar ones, so that one can study $\Mbar$ with an applicable notion of local model in the sense of \cite[Definition 3.27]{rapoport}. 
At the face of these methods $\Mbar$ might well be formally smooth of relative dimension one over $\O_K/p\O_K$, but I conjecture that the following even nicer description is valid:

\begin{conj}
\label{O3}
Let $\GO(V/K)=H/K$ be as above and let us write $K_p\subset H(\q_p\otimes K)$ for the stabilizer of some self-dual $\z_p\otimes\O_K$-lattice in $\q_p\otimes V$ and let 
$K_\infty\subset H(\r\otimes K)$ be the product of the neutral component of some maximal compact subgroup with the center, so that $K_\infty\cong\GO(3,\r)\times\SO(2,\r)\times\GO(1,\r)$.
There exists a flat, formally smooth and universally closed $\z_{(p)}\otimes\O_K$-scheme $\M$ with $H(\a^{\infty,p}\otimes K)$-action such that 
\begin{itemize}
\item[(i)]
The underlying analytic space of $\M\times_{\z_{(p)}\otimes\O_K,v}\C$ agrees $H(\a^{\infty,p}\otimes K)$-equivariantly with the Shimura curve
\begin{equation}
\label{witz}
H(K)\backslash H(\a\otimes K)/(K_\infty\times K_p)
\end{equation}
\item[(ii)]
The special fiber $\M\times_{\z_{(p)}\otimes\O_K}\O_K/p\O_K$ agrees $H(\a^{\infty,p}\otimes K)$-equivariantly with the moduli provariety $\Mbar$.
\end{itemize}
\end{conj}

The complex structure on \eqref{witz} requires a harmless choice of an identification $h:\C^\times/\r^\times\stackrel{\cong}{\rightarrow}\SO(2,\r)\subset K_\infty$, so that $(\Res_{K/\q}H^\circ,h)$ 
becomes a Shimura datum with trivial weight homomorphism. In fact it turns out to be of Abelian type, so that the methods of \cite{kisin} clearly yield an integral canonical model $\M$ over 
$\z_{(p)}\otimes\O_K$. The conjectured characterization of its special fiber in terms of polarized Abelian $6$-folds with $\O_D$-action, i.e. property (ii), is inspired by the following naive heuristic:\\

For any perfect field $k\supset\O_K/p\O_K$ one expects the $W(k)$-points on $\M$ to be given by a "$\z_{(p)}$-motive $M$" of rank $6$ over $\Spec W(k)$ which is equipped with a $\z_{(p)}\otimes\O_K$-action, 
a perfect symmetric polarization $M\times M\rightarrow\z_{(p)}(0)$, and a $\a^{\infty,p}\otimes K$-linear isometry $\eta$ between $\a^{\infty,p}\otimes V$ and the \'etale realization $M_{\mathaccent19 et}$ 
of $M$. Moreover, one expects that its de Rham realization is a $\z/2\z$-graded strongly divisible lattice $M_{dR}=M_0\oplus M_1$ in a filtered isocrystal $\q\otimes M_{dR}$ whose filtration satisfies
\begin{eqnarray*}
&\fl^{-1}M_{dR}=M_{dR}\neq\fl^0M_{dR}\supsetneq M_{dR}[1-e]=:M_0&\\
&\fl^2M_{dR}=0\neq\fl^1M_{dR}\subsetneq M_{dR}[e]=:M_1& 
\end{eqnarray*}
so that it has three steps on $M_1$ but only one single step on $M_0$. In particular $FM_0=M_1$ and the self-dual chain
$$N:=F^{-1}M_0+M_1\supset M_1\supset F^{-1}M_0\cap M_1=N^\perp\supset pN$$
of $W(k)$-lattices in $\q\otimes M_1$ satisfies $N/M_1\cong k\cong M_1/N^\perp$. Via lemma \ref{monod6} our polarized Abelian surface $\gY$ with Rosati invariant $\O_D$-action gives rise to a 
$\z/2\z$-graded symmetric Dieudonn\'e module $L_0\oplus L_1$ fulfilling $FL_0=VL_0=L_1$ and $FL_1=VL_1=pL_0$ while its pairing restricts to a perfect pairing on $L_0$. By slight abuse of notation, we 
denote the $\O_K$-linear tensor product of the special fiber of $M$ with the "motive" of $\gY$ by $M\otimes_{\O_K}\gY$. This object is a "$\z_{(p)}$-motive" of rank $12$ over $\Spec k$ which is equipped 
with a $\z_{(p)}\otimes\O_D$-action and a natural skew-symmetric polarization, and its $\z/2\z$-graded crystalline realization is $M_0\otimes L_0\oplus M_1\otimes L_1$, which is contained in a 
Dieudonn\'e lattice namely $M_0\otimes L_0\oplus N\otimes L_1$ (N.B.: $F(M_0\otimes L_0)=M_1\otimes L_1\subset N\otimes L_1$ and $F(N\otimes L_1)\subset\frac{M_0}p\otimes pL_0$). According 
to \cite[Definition 5.1]{milne}, one could expect that a Dieudonn\'e lattice in the crystalline realization of a polarized $\z_{(p)}$-motive gives rise to an isogenous $\z_{(p)}$-isogeny class of polarized 
Abelian varieties. So let $(Y,\lambda)$ be the $\z_{(p)}$-isogeny class of polarized Abelian $6$-folds with a $\z_{(p)}\otimes\O_D$-action $\iota$ whose (using lemma \ref{monod6} again) associated 
$\z/2\z$-graded symmetric Dieudonn\'e module is the lattice $M_0\otimes L_0\oplus N\otimes L_1$. Moreover, one has $\dim_k\Lie Y[1-e]=2$ and $\dim_k\Lie Y[e]=4$ and there is a natural isogeny 
$M\otimes_{\O_K}\gY\rightarrow Y$ inducing isomorphisms $M_{\mathaccent19 et}\otimes_{\a^{\infty,p}\otimes K}H_1^{\mathaccent19 et}(\gY,\a^{\infty,p})\cong H_1^{\mathaccent19 et}(Y,\a^{\infty,p})$. 

\subsection{Second moduli space}
This subsection advertises certain moduli theoretic perspectives on the construction of $G_2$-examples in section \ref{2nd}. Fix $L\supset L^+\supset\q$ and $\gq\mid\gq^+\mid p$, as in theorem \ref{G2}, and an embedding 
$K(\O_L/\gq)\stackrel{\cong}{\rightarrow}L_\gq\stackrel{\iota_0}{\rightarrow}\C$, as in subsection \ref{proof}. Recall that we fixed an auxiliary element $-v^*=v\in\O_L\backslash\{0\}$ such that \eqref{converse2} is a polarization. 
Moreover, we pick a finitely generated $*$-invariant $\z_{(p)}\otimes\O_L$-algebra $\R\subset\z_{(p)}\otimes\O_L^{\oplus2}$ satisfying $\R[\frac1p]=L^{\oplus2}$, along with a subset $\Omega\subsetneq\z/r\z$ of 
cardinality seven and consider the $W(\O_L/\gq)$-scheme $\gM_{\R,\Omega}$ parameterizing quintuples $(Y^{(1)},\lambda^{(1)},\iota^{(1)},\iota,\eta^{(1)})$ over connected $W(\O_L/\gq)$-schemes $S$ with the following properties:

\begin{itemize}
\item
$Y^{(1)}\rightarrow S$ is an Abelian $7r$-fold up to $\z_{(p)}$-isogeny and $\lambda^{(1)}$ is a homogeneous class of polarizations on 
$Y^{(1)}$ allowing a $p$-principal representative (thus inducing a skew-symmetric self-duality $(Y^{(1)}[p^\infty])^t\cong Y^{(1)}[p^\infty]$).
\item
$\iota^{(1)}:\O_L\rightarrow\z_{(p)}\otimes\End(Y^{(1)})$ is a Rosati-invariant action such hat the $F^\sigma\circ\iota_\gq$-eigenspace of $\Lie Y$ is an invertible 
$\O_S$-module (resp. vanishes) whenever $\sigma\in\Omega$ (resp. $\sigma\notin\Omega$), where $F:K(\O_L/\gq)\rightarrow K(\O_L/\gq)$ denotes the absolute Frobenius.
\item
$\eta^{(1)}:\a^{\infty,p}\otimes\8_0\otimes L\stackrel{\cong}{\rightarrow}H_1^{\mathaccent19 et}(Y_s^{(1)},\a^{\infty,p})$ is a $\a^{\infty,p}\otimes L$-linear and $\pi_1^{\mathaccent19 et}(S,s)$-invariant
similitude, where the skew-Hermitian pairing on $L$ is defined by $\psi(x,y)=\tr_{L/\q}(vxy^*)$, the euclidean $\q$-space of purely imaginary octonions is denoted by $\8_0$ and $s$ is an arbitrary 
geometric point on $S$.
\item
$\iota:\R\rightarrow\z_{(p)}\otimes\End_L(Y^{(2)})$ is a Rosati-invariant $\O_L$-linear action, where $(Y^{(2)},\lambda^{(2)},\iota^{(2)})$ denotes the (canonically homogeneously 
$p$-principally polarized) 2nd exterior power of $(Y^{(1)},\lambda^{(1)},\iota^{(1)})$, which is formed with the help of the catalyst $Y^{(0)}$ (cf. \cite[Theorem 4.8]{habil}). Moreover, we also 
request the $\R$-linearity of the level structure
$$\a^{\infty,p}\otimes(\bigwedge^2\8_0)\otimes L\stackrel{\cong}{\rightarrow}H_1^{\mathaccent19 et}(Y_s^{(0)},\a^{\infty,p})\otimes_{\a^{\infty,p}\otimes L}H_1^{\mathaccent19 et}(Y_s^{(2)},\a^{\infty,p}),$$
which is naturally inherited from $\eta^{(1)}$.
\end{itemize}

Just as in subsection \ref{proof} we write $\Auto(\8)=:G/\q$ for the $\r$-compact $\q$-form of the simple algebraic group of type $G_2$. Let us define a $\q$-torus $Z$ as the kernel of 
$$\g_m\times\Res_{L/\q}\g_m\rightarrow\Res_{L^+/\q};(t,a)\mapsto taa^*,$$ 
and notice that $Z(\r)\cong\C^\times\times\SO(2,\r)^{r-1}$ and that $Z(\a^{\infty,p})\times G(\a^{\infty,p}\otimes L^+)$ acts on $\gM_{\R,\Omega}$, of which the generic fiber is empty. So what can 
be said about the $Z(\a^{\infty,p})\times G(\a^{\infty,p}\otimes L^+)\times\Gal(\fbar_p/\f_{p^r})$-representations $H_{\mathaccent19 et}^i(\gM_{\R,\Omega}\times_{W(\O_L/\gq)}\fbar_p,\qbar_\ell)$
or the $Z(\a^{\infty,p})\times G(\a^{\infty,p}\otimes L^+)\times\Gal(\fbar_p/\f_{p^r})$-set $\gM_{\R,\Omega}(\fbar_p)$? 
In this direction it seems reasonable to try to replace the special fiber $\gM_{\R,\Omega}\times_{W(\O_L/\gq)}\O_L/\gq$ by a smaller and more canonical variety, in the spirit of the following:

\begin{conj}
\label{ShG2} 
If $\R$ is sufficiently small, then the special fiber of $\gM_{\R,\Omega}$ contains a $Z(\a^{\infty,p})\times G(\a^{\infty,p}\otimes L^+)$-invariant closed subvariety $\emptyset\neq\gm$ 
which is smooth of dimension $6$ over $\O_L/\gq$ such that  the formal isogeny types of $Y^{(1)}[\gq^\infty]$ at any points on $\gm$ with values in a perfect field are one of 
\begin{itemize}
\item
$G_{0,1}^{2r}\oplus G_{1,r-1}^3\oplus G_{2,r-2}^2$
\item
$G_{0,1}^r\oplus G_{1,2r-1}\oplus G_{1,r-1}\oplus G_{3,2r-3}\oplus G_{2,r-2}$
\item
$G_{1,2r-1}\oplus G_{1,r-1}^3\oplus G_{3,2r-3}$
\item
$G_{1,r-1}^7$
\end{itemize}
giving rise to a dense open Newton stratum, two locally closed equidimensional Newton strata of dimension $5$ and $4$ and a closed non-equidimensional Newton stratum whose 
irreducible components have the dimension $5$ or $3$. Furthermore, $G(\a^{\infty,p}\otimes L^+)$ acts trivially on the set of irreducible components of $\gm\times_{\O_L/\gq}\fbar_p$.
\end{conj}

Our guesses on the dimensions of the Newton strata stand in line with their (by \cite{deJong}) known purity and with known properties of affine Deligne-Lusztig varieties 
(cf. \cite[Conjecture 1.0.1]{GHKR}). The expected occurrence of four Newton strata stems from their Newton cocharacters factoring through $Z\times\Res_{L^+/\q}G$.


\begin{thebibliography}{1}

\bibitem{springer}
F. van der Blij, T.A. Springer, {\em The arithmetic of octaves and the group G2}, Indag. Math. 21(1959), p. 406–418.

\bibitem{height3/2}
O. B\"ultel, {\em Endomorphism rings of some Abelian varieties and an application}, talk in the 'Oberseminar Modulfunktionen' University of Heidelberg, 15.11.2000

\bibitem{habil}
O. B\"ultel, {\em Construction of Abelian varieties with given monodromy}, Geom. funct. anal. 15(2005), p.634-696

\bibitem{E7}
O. B\"ultel, {\em PEL Moduli Spaces without $\C$-valued points}, arXiv preprint

\bibitem{cadoret}
A. Cadoret, A. Tamagawa, {\em Ghosts in families of Abelian varieties with a common isogeny factor}, to appear in JEMS 

\bibitem{deJong}
A.J. de Jong, F. Oort, {\em Purity of the stratification by Newton polygons}, J. Amer. Math. Soc. 13(2000), p.209-241

\bibitem{deligne4}
P. Deligne, {\em Travaux de Shimura}, S\'eminaire Bourbaki Expos\'es 382-399, Lecture Notes in Math. 244, Springer-Verlag 1971, p.123-165

\bibitem{deligne3}
P. Deligne, {\em Vari\'et\'es de Shimura: interpr\'etation modulaire, et techniques de construction de mod\'eles canonique}, Automorphic forms, 
representations, and $L$-functions, Proc. Sympos. Pure Math. 33,2 (ed Borel and Casselman, Amer. Math. Soc., Providence RI, 1979), p.247-290

\bibitem{deligne2}
P. Deligne, {\em Hodge cycles on abelian varieties}, Hodge Cycles, Motives, and Shimura varieties, Lecture Notes in Math. 900, Springer 1982, p.9-100

\bibitem{emerton}
M. Emerton, {\em A $p$-adic variational Hodge conjecture and modular forms with complex multiplication}, preprint

\bibitem{GHKR}
U. G\"ortz, T.J. Haines, R.E. Kottwitz, D.C. Reuman, {\em Dimensions of some affine Deligne-Lusztig varieties}, Annales scientifiques de l'\'Ecole Normale Sup\'erieure, 39(2006), p.467-511

\bibitem{H1}
M. H. Hedayatzadeh, {\em Exterior powers of $\pi$-divisible modules over fields}, J. Number Theory 138(2014), p.119–174

\bibitem{H2}
M. H. Hedayatzadeh, {\em Exterior powers of Lubin-Tate groups}, J. Th{\'e}or. Nombres Bordeaux 27(2015), p.77-148
	
\bibitem{humphreys}
J.E. Humphreys, {\em Introduction to Lie Algebras and Representation Theory}, Graduate Texts in Mathematics 9, Springer-Verlag New York 1972

\bibitem{katz2}
N. Katz, {\em Serre-Tate Local Moduli}, Surfaces Alg\'ebriques, Springer Lecture Notes in Math 868, p.138-202

\bibitem{katz}
N. Katz, {\em Exponential Sums and Differential Equations}, Annals of Mathematics Studies 124, Princeton University Press 1990

\bibitem{kisin}
M. Kisin, {\em Integral models for Shimura varieties of abelian type}, J. Amer. Math. Soc. 23 (2010), p.967-1012

\bibitem{kottwitz}
R.E. Kottwitz, {\em Points on some Shimura varieties over finite fields}, Journal of the AMS 5(1992), p.373-444

\bibitem{meszin}
W. Messing, T .Zink, {\em De Jong's Theorem on Homomorphisms of p-divisible Groups}, unpublished preprint January 8, 2002

\bibitem{milne}
J.S. Milne, N. Ramachandran, {\em Integral Motives and Special Values of Zeta Functions}, preprint arXiv:math/0204065
 
\bibitem{mori}
L. Moret-Bailly, {\em Pinceaux de Vari\'et\'es ab\'eliennes}, Ast\'erisque 129(1985)

\bibitem{mumford}
D. Mumford, {\em Abelian varieties}, TIFR Bombay, Oxford University Press 1970

\bibitem{pink1}
R. Pink, {\em $\ell$-adic algebraic monodromy groups, cocharacters, and the Mumford-Tate conjecture}, J. reine angew. Math. 495(1998), p.187-237

\bibitem{pink2}
R. Pink, {\em The Galois representations associated to a Drinfeld module in special characteristic I: Zariski density}, Journal of Number Theory 116(2006), p.324-347

\bibitem{oort3}
F. Oort, {\em Endomorphism algebras of Abelian varieties}, Algebraic Geometry and Commutative Algebra in honor of 
Masayoshi Nagata (Ed. H. Hijikata et al), Kinokuniya Company Ltd, Tokyo, 1988, volume 2, p.469-502

\bibitem{oort2}
F. Oort, {\em Some questions in algebraic geometry}, preliminary version, Utrecht University, June 1995

\bibitem{oort4}
F. Oort, J. Tate, {\em Group schemes of prime order}, Ann. scient. \'Ec. Norm. Sup. quatri\`eme s\'erie, tome 3 (1970), p. 1-21

\bibitem{oort1}
F. Oort, M. van der Put, {\em A construction of an abelian variety with a given endomorphism algebra}, Compos. Math. 67(1988), p.103-120

\bibitem{rapoport}
M. Rapoport, T. Zink, {\em Period Spaces for $p$-divisible Groups}, Annals of Mathematics Studies 141, Princeton University Press 1996

\bibitem{satake}
I. Satake, {\em Algebraic Structures of Symmetric Domains}, Publications of the Mathematical Society of Japan 14, Iwanami Shoten 1980

\bibitem{tate}
J. Tate, {\em Classes d'isog\'enie des vari\'et\'es ab\'eliennes sur un corps fini}, S\'eminaire Bourbaki Expos\'es 347-363 
Lecture Notes in Mathematics number 179 (Springer- Verlag, Berlin, 1971), p.95-110

\bibitem{wedhorn}
T. Wedhorn, {\em Ordinariness in good reductions of Shimura varieties of PEL-type} Ann. scient. \'Ec. Norm. Sup. quatri\`eme s\'erie, tome 32 (1999), p.575-618

\bibitem{zarhin}
Y.G. Zarhin, {\em Weights of simple Lie algebras in the cohomology of algebraic varieties}, Math. U.S.S.R. Izvestiya 24 (1985), p.245-281

\bibitem{zink3}
T. Zink, {\em A Dieudonn\'e theory for $p$-divisible groups}. In: Class field theory - its centenary and prospect p.139-160, 
Adv. Stud. Pure Math. 30, Math. Soc. Japan 2001

\bibitem{zink1}
T. Zink, {\em Windows for Displays of $p$-divisible Groups}, Moduli of abelian varieties, Progress in Mathematics 195, Birkh\"auser Verlag 2001, p.491-518

\end{thebibliography}
\end{document}